\documentclass{amsart}

\usepackage[colorlinks=true, allcolors=blue]{hyperref}
\usepackage{amsmath,amsfonts,amssymb}
\usepackage{amsthm}

\newtheorem{theorem}[subsection]{Theorem}
\newtheorem{lemma}[subsection]{Lemma}
\newtheorem{cor}[subsection]{Corollary}
\newtheorem{remark}[subsection]{Remark}
\newtheorem{example}[subsection]{Example}
\newtheorem{definition}[subsection]{Definition}

\title[Invariants of $\text{O}_{2m+1}(\mathbb{F}_{2^s})$]{Invariants of the finite orthogonal groups in odd dimension and even characteristic}
\author{H.E.A. Campbell, R.J. Shank and D.L. Wehlau}
\date{\today}

\address{Department of Mathematics and Statistics, Queen's 
University, Kingston ON, Canada}
\email{eddy@unb.ca}
\address{School of Engineering, Mathematics and Physics, 
University of Kent,
Canterbury, United Kingdom}
\email{R.J.Shank@kent.ac.uk}
\address{Department of Mathematics and Computer Science, Royal Military College of Canada, Kingston ON, Canada}
\email{wehlau@rmc.ca}



\newcommand{\field}{\mathbb{F}}
\newcommand{\gl}[2]{{{\text{GL}_{{#1}}(\field_{#2})}}}
\newcommand{\SL}[2]{{{\text{SL}_{{#1}}(\field_{#2})}}}
\newcommand{\syp}[2]{{{\text{Sp}_{{#1}}(\field_{#2})}}}
\newcommand{\orthp}[2]{{{\text{O}^{+}_{{#1}}(\field_{#2})}}}

\newcommand{\orth}[2]{{{\text{O}_{{#1}}(\field_{#2})}}}

\newcommand{\PP}{\mathcal{P}}

\newcommand{\V}{\mathcal{V}}

\newcommand{\HH}{\mathcal{H}}

\DeclareMathOperator{\ns}{ns}
\newcommand{\FF}{\mathcal{Q}}

\DeclareMathOperator{\lc}{LC}
\newcommand{\h}{\mathcal{H}}
\renewcommand{\H}{\mathcal{H}}
\newcommand{\frob}{\mathcal{F}}
\newcommand{\basise}{\lambda}
\newcommand{\basisf}{\mu}

\DeclareMathOperator{\rad}{rad}
\DeclareMathOperator{\Span}{Span}
\DeclareMathOperator{\lt}{LT}

\begin{document}
\begin{abstract}
    We describe the ring of invariants for the finite orthogonal groups
in odd dimension and even characteristic acting on the defining representation.  
We construct a minimal algebra generating set and describe the relations among the generators.  This ring of invariants is shown to be 
a complete intersection and thus is
Cohen-Macaulay.  
This extends the previous computation 
of Kropholler, Mohseni Rajaei, and Segal
valid over the field of order 2.
\end{abstract}

\maketitle
\tableofcontents

\section{Introduction}

The fundamental problem in the invariant theory of finite groups is to determine the ring of invariants 
of a representation of a finite group. 
Over a field of characteristic zero, this problem is reasonably well understood; see the excellent survey article by 
Stanley~\cite{stanley-invafinigrouthei:79}.
In positive characteristic the situation is more complex.
If the order of the group is a unit in the field then many of the characteristic zero methods still work. However for modular representations,
i.e., when the characteristic of the field divides the order of the group, new methods and ideas are needed; see 
\cite{benson-polyinvafinigrou:93},
\cite{campbell+wehlau:mit11},
\cite{derksen+kemper-invalggp:08}
or
\cite{neusel+smith-invatheofinigrou:02}.
The defining representations of the finite classical groups provide interesting families of modular representations.
While almost all of the defining representations for these groups are generated by pseudo-reflections, the rings of invariants are rarely polynomial rings. 
In 1911, L.E.~Dickson~\cite{dickson-fundsystinvagene:11} gave an explicit description of the ring of 
invariants of the general linear group over any finite field.   
The rings of invariants for the symplectic groups were computed by David Carlisle and Peter Kropholler in the 1990s, see 
\cite[\S 8.3]{benson-polyinvafinigrou:93}.
The invariants for the finite unitary groups were computed by 
Huah Chu and Shin-Yao Jow \cite{chu+jow-unitary:06}.
For the orthogonal groups there is no general result.  
Let $\field_q$ denote the finite field of order $q$.
Kropholler, Mohseni Rajaei and Segal computed the invariants for
all orthogonal groups defined over $\field_2$ \cite{Kropholler+Rajaei+Segal:05}.
In \cite{campbell+shank+wehlau:24}, we compute the invariants for orthogonal groups of plus type and odd characteristic.
Various small dimensional cases have been computed by
Chiang and Hung \cite{chiang+hung-invorth:93},
Chu \cite{chu-polyinvfourdimorth:01},
Cohen \cite{cohen-ratfuncortho:90}
 and Smith \cite{smith-ringinvorth:99}.
All of these rings of invariants are complete intersections and, therefore, Cohen-Macaulay.
In this paper we  compute the ring of invariants for the defining representation of $\orth{2m+1}{q}$ for $q=2^s>2$. The group is determined by the quadratic form $\xi_0$. Applying Steenrod operations to $\xi_0$ produces invariants
$\xi_i$ of degree $q^i+1$ for $i>0$. 
We construct an invariant $e_1$ of degree $q^{2m}(q-1)/2$.
Applying Steenrod operations to $e_1$ produces invariants $e_i$
of degree $q^{2m}(q^i-1)/2$ for $i>1$. We show  that the ring of invariants
is the complete intersection generated by 
$\{\xi_0,\xi_1,\ldots,\xi_{2m-1}\}\cup\{e_1,\ldots,e_m\}$
subject to relations which rewrite $\xi_{2m-i}^{q^i/2}$ for $i<m$.  
We also show that
$\h:=\{\xi_0,\xi_1,\ldots,\xi_m,e_1,\ldots, e_m\}$ is a homogeneous system of parameters and that the invariant ring is the free module over the algebra generated by $\h$ with basis given by the monomial factors of 
$\prod_{i=1}^{m-1} \xi_{2m-i}^{(q^i/2)-1}$.
We note that the ring of invariants is generated by $\{\xi_0,e_1\}$
as an algebra over the Steenrod algebra. We conjecture that
for the defining representation of any finite classical group,
the ring of invariants is generated by at most two elements  as an algebra over the Steenrod algebra. The conjecture has been verified for the general linear groups, the special linear groups, the symplectic groups and
the orthogonal groups of plus type in odd characteristic.

In Section~\ref{sec: preliminaries} we introduce the problem and the main tools 
including the definition of the Steenrod operations.  
Section~\ref{sec: dicksons} introduces the Dickson invariants, which 
generate the invariants of the general linear group.
In Section~\ref{sec: symplectics} we recall the computation of the invariants for the symplectic group and derive some results special
to characteristic 2.
Section~\ref{orth_invs} defines the orthogonal invariants $e_i$ and develops some of their properties.
In Section~\ref{sec: dim7} we compute the invariants for the group
$\orth{7}{q}$ as a clarifying example illustrating our techniques.
Section~\ref{sec: relations} develops and describes the relations among
our generators for the invariants of $\orth{2m+1}{q}$.
Finally in Section~\ref{sec: ufd} we complete the proof of the main theorem
(Theorem~\ref{main_theorem}).

\section{Preliminaries}\label{sec: preliminaries}

For a vector space $V$, the right action of ${\rm GL}(V)$
on $V$ induces a left action on the dual $V^*$ given by
$(\phi \cdot g)(v)=\phi(g\cdot v)$ for 
$\phi\in V^*$, $g\in {\rm GL}(V)$ and $v\in V$.
The action on $V^*$extends to an action by
algebra automorphisms on the symmetric algebra of $V^*$.
Choosing a basis for $V^*$ allows us to identify the symmetric algebra
of $V^*$ with the polynomial algebra generated by the basis elements.
 In this paper we work over the field $\field_q$ where $q=2^s$.
We study the ring of invariants of the orthogonal group $\orth{2m+1}{q}$ of order $q^{m^2}\prod_{j=1}^m(q^{2j}-1)$
(see \cite[page 81]{Taylor-ClassicalGroups:92}). 
Because we work in characteristic 2, signs are irrelevant.   However, we
choose to use minus signs in certain places to improve the readability 
of some formulae.

Consider the polynomial algebra $S=S_m=\field_q[y_1,\ldots,y_m,x_m,\ldots,x_1]$.
Define $$\xi_0:=z^2+x_1 y_1+x_2y_2+\cdots + x_m y_m\in S[z].$$
We use the ordered basis for $V^*$ given by
    $[y_1,\ldots, y_m, z, x_m,\ldots, x_1]$
with dual basis
$[\basise_1,\ldots, \basise_m, \omega, \basisf_m,\ldots, \basisf_1]$
for $V$.
The group $\orth{2m+1}{q}$ is the subgroup of $\gl{2m+1}{q}$ which fixes $\xi_0$. 
The associated bilinear form, $B$, is alternating, symmetric, degenerate and does not determine the quadratic form 
(see \cite[page 142]{Taylor-ClassicalGroups:92}). 
This bilinear form is given by 
$B(u,v)=\xi_0(u+v)+\xi_0(u)+\xi_0(v)$.
The matrix representing $B$ using our chosen basis is
the $(2m+1)\times(2m+1)$ matrix
$$
\renewcommand{\arraystretch}{0.8} 
\left(\begin{array}{c|c|c}
&0&\\[-4pt]
\mbox{\normalfont\large\bfseries 0}_m&
\vdots&\mbox{\normalfont\large\bfseries J}_m\\
&0&\\
 \hline
0 \cdots0&0&0 \cdots 0\vphantom{a^1}\\
\hline
&0&\vphantom{a^1}\\[-4pt]
\mbox{\normalfont\large\bfseries J}_m
&\vdots&\mbox{\normalfont\large\bfseries 0}_m\\
&0&\\
\end{array}\right)
$$
where $\mbox{\normalfont\large\bfseries 0}_m$ is the $m\times m$ zero matrix and
\makeatletter
\def\Ddots{\mathinner{\mkern1mu\raise\p@
\vbox{\kern7\p@\hbox{.}}\mkern2mu
\raise4\p@\hbox{.}\mkern2mu\raise7\p@\hbox{.}\mkern1mu}}
$$\mbox{\normalfont\large\bfseries J}_m:=\begin{pmatrix}
    0&0&\cdots& 0&1\\
    0&0&\cdots& 1&0\\[-5pt]
    \vdots&\vdots&\Ddots&\vdots&\vdots\\
    0&1&\cdots&0&0\\
    1&0&\cdots&0&0
\end{pmatrix}.$$

The {\it radical} of $B$ is
$$\rad(B):=\{u\in V\mid B(u,v)=0,\,\forall\, v\in V\}=\omega\field_q$$
(the radical of the bilinear form is one dimensional with basis vector $\omega$ dual to $z$) and the {\it radical} of $\xi_0$ is
$$\rad(\xi_0):=\{u\in \rad(B)\mid \xi_0(u)=0\}=\{0\}.$$
Therefore, $B$ is degenerate but $\xi_0$ is non-degenerate.
Note that $\rad(B)$ is an $\orth{2m+1}{q}$ submodule of $V$
and $$\left(V/\rad(B)\right)^*=\Span_{\field_q}\{y_1,\ldots,y_m,x_m,\ldots,x_1\}$$
is an $\orth{2m+1}{q}$ submodule of $V^*$.
Furthermore,
the restriction of $\orth{2m+1}{q}$ to $\left(V/\rad(B)\right)^*$
(and to $S$) is faithful and is the usual action of the symplectic group $\syp{2m}{q}$ 
(see \cite[Theorem 11.9]{Taylor-ClassicalGroups:92}).

The {\it complete Steenrod operator} $\PP(t):S[z]\to S[z,t]$ is the algebra homomorphism determined by $\PP(t)(v)=v+v^q t$ for $v$ homogeneous of degree one. Since the map is linear in degree one, $\PP(t)$ is well-defined. For $f$ homogeneous of degree $d$, the Steenrod operations 
$\PP^i(f)$ are defined by
$$\PP(t)(f)=\sum_{i=0}^d\PP^i(f)t^i.$$
Note that for $i>d$ or $i<0$, $\PP^i(f)=0$.
It is clear that $\PP^0(f)=f$ and $\PP^d(f)=f^q$,
i.e., the {\it stability} property is satisfied.
The Steenrod operations satisfy the {\it Cartan identity}:
for $f_1,f_2\in S[z]$
$$\PP^i(f_1f_2)=\sum_{j=0}^i \PP^{j}(f_1)\PP^{i-j}(f_2).$$
The Steenrod operations also satisfy the {\it Adem relations}: 
for $i<qj$
$$\PP^i\PP^j=\sum_k (-1)^{i+k}
\binom{(q-1)(j-k) -1}{i-qk}\PP^{i+j-k}\PP^{k}.$$
We can extend the action of $\gl{2m+1}{q}$ to $S[z,t]$
by taking $tg=t$ for all $g$. Using this action,
since taking a $q^{th}$ power is linear in $S[z]$, 
we see that $\PP(t)$ commutes with the $\gl{2m+1}{q}$-action
and $\PP^ig=g\PP^i$ for all $i$.

The following lemma is a consequence of the Cartan identity.

\begin{lemma} For $f\in S[z]$, we have $\PP^i(f^q)=0$ unless $q$ divides $i$,
in which case $\PP^i(f^q)=(\PP^{i/q}(f))^q$.
\end{lemma}

\begin{lemma}\label{steenrod on linear forms}
    Suppose $v,f\in S[z]$ with $v$ homogeneous of degree one.
    Then $v$ divides $\PP^i(vf)$ for all $i$.
\end{lemma}
\begin{proof}
    By definition, $v$ divides $\PP^j(v)$. Therefore, using the Cartan identity, $v$ divides $\PP^i(vf)$.
\end{proof}

It is an immediate consequence of Lemma~\ref{steenrod on linear forms} that if $f$ is a product of linear forms, $f$ divides $\PP^i(f)$.
Define $\overline{\xi}_0:=x_1y_1+x_2y_2+\cdots +x_my_m$
and, for $i>0$, $$\xi_i:=\sum_{j=1}^m (x_jy_j^{q^i}+y_jx_j^{q^i}).$$

 \begin{cor} \label{com_st}(a) $\PP(t)(\xi_0)=\xi_0+\xi_1 t+\xi_0^q t^2$.\\
 (b) $\PP(t)(\xi_1)=\xi_1+2\overline{\xi}_0^q t+\xi_2 t^q+\xi_1^q t^{q+1}$.\\
 (c) For $i>1$, 
 $\PP(t)(\xi_i)=\xi_i+\xi_{i-1}^q t+\xi_{i+1} t^{q^i}+\xi_i^q t^{q^i+1}$.
 \end{cor}

Then we have   
$$\PP^1(\xi_0)=x_1 y_1^q+x_1^qy_1+\cdots+x_m y_m^q+x_m^qy_m=\xi_1\in S^{\syp{2m}{q}}\subset S[z]^{\orth{2m+1}{q}}.$$
and $\PP^1(\xi_1)=2\overline{\xi}_0^q=0$.
Also note that the point-wise stabiliser of $z$ in $\orth{2m+1}{q}$ is isomorphic to $\orthp{2m}{q}$.

\begin{lemma} \label{terminal_var}
For all $g\in\orth{2m+1}{q}$, $S[z](g-1)\subset S$. 
\end{lemma}
\begin{proof} 
Since $\rad(B)={\rm Span}_{\field_q}\{\omega\}$ and 
$g(\rad(B))=\rad(B)$, we have $g(\omega)=\gamma \omega$ for some $\gamma\in\field_q$. 
However $\xi_0(\omega)=1$ and $\xi_0(\omega)=\xi_0(g(\omega))=\gamma^2$. Therefore $\gamma=1$ and $g(\omega)=\omega$.
The result then follows from the fact that $z$ is dual to $\omega$.
\end{proof}

Consider $F\in S[z]$. Let $\lc_z(F)\in S$ denote the leading coefficient of $F$ as a polynomial in $z$ with coefficients in $S$. 
Since $\lc_z(\xi_0)=1$, we can divide $F$ by $\xi_0$ to get a quotient $f\in S[z]$ and a remainder $az+b$ with $a,b\in S$.
The following is a consequence of Lemma~\ref{terminal_var}.

\begin{lemma} \label{quotient_lem}
For $G$ a subgroup of $\orth{2m+1}{q}$, if $F\in S[z]^G$ then $\lc_z(F)\in  S^G$.
Furthermore, if $F=f\xi_0 +az+b$ with $a,b\in S$, then both $f$ and $az+b$ are elements of $S[z]^G$.
\end{lemma}

It follows from Lemma~\ref{quotient_lem} that $S[z]^G$ is generated by $\xi_0$, elements of $S^G$, together with elements of the form $az+b$ with $a\in S^G$ and $b\in S$.
Furthermore, since $(az+b)^2-a^2\xi_0=a^2\overline{\xi}_0+b^2\in S^G$, we see that $b^2\in S^{G_z}$, where $G_z$ is the point-wise stabiliser of $z$ in $G$.
 This means that $b\in S^{G_z}$. Note that $G_z= G\cap \orthp{2m}{q}$.

Suppose $f=az+b\in S[z]^G$ for $G$ a subgroup of $\orth{2m+1}{q}$ and $a,b\in S$.
Using Lemma~\ref{quotient_lem}, $a\in S^G$. Squaring and eliminating $z$ gives 
$$f^2+a^2\xi_0=b^2+a^2\overline{\xi}_0\in S^G.$$ 
For which $F\in S^G$ can we find $f=az+b\in S[z]^G$ such that $F=f^2+a^2\xi_0$?
With this question in mind, for a polynomial $F\in S$ define $\ns(F)$ to be the sum of the {\it non-square terms} of $F$.
Note that since every element of $\field_q$ is a square, a term is a square if and only if the associated monomial is a square.
If $F\in S^G$ and $\ns(F)=a^2\overline{\xi}_0$ with $a \in S^G$, then $b^2=F+\ns(F)=F+a^2\overline{\xi}_0$
determines $b$ and
$$f^2+a^2\xi_0=b^2+a^2\overline{\xi}_0=F.$$

\begin{lemma} \label{inv-cons-lem} Suppose $G\leq\orth{2m+1}{q}$ and $F\in S^G$ with $\ns(F)=a^2\overline{\xi}_0$ for some $a\in S^G$. Then
$b^2:=F+\ns(F)$ determines $b\in S^{G_z}$ and $f:=az+b\in S[z]^G$.
\end{lemma}
\begin{proof}
Clearly $b^2=F+\ns(F)$ determines $b\in S$.
Since $\ns(F)=a^2\overline{\xi}_0\in S^{G_z}$, we have $b^2\in S^{G_z}$.
Hence $b\in S^{G_z}$. Since $f^2=a^2\xi_0+F$, we see that $f^2\in S[z]^G$.
Thus $f\in S[z]^G$.
\end{proof}

\section{Dickson Invariants}\label{sec: dicksons}

The Dickson invariants are a generating set for the ring of invariants of the general linear group over a finite field
(see \cite[\S 8.1]{benson-polyinvafinigrou:93},  
\cite[\S 3.3]{campbell+wehlau:mit11} or
\cite{wilkerson-primdickinva:83}).
We use $d_{i,m}$ to denote the Dickson invariants for
the action of $\gl{2m}{q}$ on $S$ and let $\widetilde{d}_{i,n}$ denote the Dickson invariants for the action of
$\gl{n}{q}$ on $\field_q[x_1,x_2,\ldots,x_n]$. 
Note that $\widetilde{d}_{i,2m}$ is $d_{i,m}$ up to a relabelling of the variables. Similarly, take $u_m=\prod_{i=1}^mN(y_i)N(x_i)$ and $\widetilde{u}_n=\prod_{i=1}^n N(x_i)$,
where  $N$ denotes the orbit product over the upper triangular unipotent subgroup of the appropriate general linear group.
Note that $\widetilde{d}_{n,n}$ is the orbit product of $x_1$ over 
$\gl{n}{q}$ and $\widetilde{d}_{n,n}=\widetilde{u}_n^{q-1}$.
We also have a matrix description
$$\widetilde{u}_n=\det
\begin{pmatrix}
x_1&x_2&x_3&\cdots&x_n\\
x_1^q&x_2^q&x_3^q&\cdots & x_n^q\\
&&\vdots&&\\
x_1^{q^{n-1}}&x_2^{q^{n-1}}&x_3^{q^{n-1}}&\cdots & x_n^{q^{n-1}}
\end{pmatrix}.$$
For a monomial $\beta\in S$, let $\sigma(\beta)$ denote the orbit sum over the symmetric group on the variables appearing in $\beta$, 
the so-called monomial symmetric function associated to $\beta$.
Using the matrix descriptions of the $\SL{n}{q}$-invariants
(see \cite{wilkerson-primdickinva:83}), since we are in characteristic two, $\widetilde{u}_n=\sigma(x_1x_2^q\cdots x_n^{q^{n-1}})$
and, for $i<n$,  $$\widetilde{u}_n\widetilde{d}_{i,n}=\sigma(x_1x_2^q\cdots x_{n-i}^{q^{n-i-1}} x_{n-i+1}^{q^{n-i+1}}\cdots x_n^{q^n}).$$

\begin{lemma} \label{stn_un} For $0<k<q^{n-1}$, $\PP^k(\widetilde{u}_n)=0$.
\end{lemma}
\begin{proof} Let $\psi:\field_q[x_1,\ldots,x_n]\to \field_q[x_1,\ldots,x_n][t]$ denote the algebra homomorphism determined by $\psi(v)=v^q-vt^{q-1}$ for $\deg(v)=1$.
For $f$ homogeneous of degree $d$, by comparing $\psi$ and the complete Steenrod operator $\PP(t)$, we see that
$$\psi(f)=\sum_{\ell=0}^d\PP^{d-\ell}(f)(-t^{q-1})^{\ell}.$$
If $v$ is a linear factor of $\widetilde{u}_n$, then the roots of $\psi(v)$ are the non-zero scalar multiples of $v$. 
From this, using the fact that $\psi(\widetilde{u}_n)/\widetilde{u}_n$ is monic of degree $(q-1)\deg(\widetilde{u}_n)$,
 we conclude that
$$\psi(\widetilde{u}_n)/\widetilde{u}_n=\prod\{t-x_1 g\mid g\in \gl{n}{q}\}=t^{q^n-1}+\sum_{i=1}^n \widetilde{d}_{i,n}t^{q^{n-i}-1}.$$ 
From this we conclude $\PP^k(\widetilde{u}_n)$ is either zero or $\widetilde{u}_n \widetilde{d}_{i,n}$ for some $i\in\{0,\ldots,n\}$.
Therefore the first non-trivial Steenrod operation  is $\PP^{q^{n-1}}(\widetilde{u}_n )=\widetilde{u}_n \widetilde{d}_{1,n}$ .
\end{proof}

From \cite[Prop. 1.3]{wilkerson-primdickinva:83}, 
we have the induction formula 
$$\widetilde{d}_{i,n}=\widetilde{d}_{i,n-1}^q+\widetilde{d}_{i-1,n-1}N(x_n)^{q-1}$$
where $N(x_n)=x_n^{q^{n-1}}+ \widetilde{d}_{1,n-1}x_n^{q^{n-2}}+\cdots +\widetilde{d}_{n-1,n-1}x_n$.
Note that $\widetilde{d}_{1,1}=x_1^{q-1}$. Therefore $\widetilde{d}_{1,2}=x_1^{q(q-1)}+N(x_2)^{q-1}=x_1^{q(q-1)}+N(x_2)^{q-2}(x_2^q+x_2x_1^{q-1})$.
Hence $\ns(\widetilde{d}_{1,2})=(N(x_2)x_1)^{q-2}(x_1x_2)=\widetilde{u}_2^{q-2}x_1x_2$. 
Since 
$$\widetilde{d}_{n,n}=\widetilde{u}_n^{q-1}=\widetilde{u}_n^{q-2}\sigma(x_1x_2^q\cdots x_n^{q^{n-1}}),$$ we have $\ns(\widetilde{d}_{n,n})=\widetilde{d}_{n,n}$.

\begin{lemma}\label{ns_di_lem} For $n\geq 2$ 
$$\ns(\widetilde{d}_{1,n})=\widetilde{u}_n^{q-2}\sigma(x_1x_2x_3^q x_4^{q^2} \cdots \, x_n^{q^{n-2}})$$
and, for $1<i<n$, 
$$\ns(\widetilde{d}_{i,n})=\widetilde{u}_n^{q-2}\sigma(x_1x_2x_3^q\cdots x_{n-i+1}^{q^{n-i-1}}x_{n-i+2}^{q^{n-i+1}}\cdots x_n^{q^{n-1}}).$$
\end{lemma}
\begin{proof} The proof is by induction on $n$. For $n=2$, we have $\ns(\widetilde{d}_{1,2})=\widetilde{u}_2^{q-2}x_1x_2$ with $\sigma(x_1x_2)=x_1x_2$.
For $n>2$, the proof is by induction on $i$. For $i=1$, the induction formula gives
$$\widetilde{d}_{1,n}=\widetilde{d}_{1,n-1}^q+N(x_n)^{q-1}=\widetilde{d}_{1,n-1}^q+N(x_n)^{q-2}(x_n^{q^{n-1}}+\sum_{j=1}^{n-1}\widetilde{d}_{j,n-1} x_n^{q^{n-j-1}}).$$
Therefore
$$\ns(\widetilde{d}_{1,n})=N(x_n)^{q-2}(\ns(x_n\widetilde{d}_{n-1,n-1})+\sum_{j=1}^{n-2}\ns(\widetilde{d}_{j,n-1}) x_n^{q^{n-j-1}}).$$
Note that $\ns(\widetilde{d}_{n-1,n-1})=\widetilde{d}_{n-1,n-1}=\widetilde{u}_{n-1}^{q-2}\sigma(x_1x_2^q\cdots x_{n-1}^{q^{n-2}})$.
For $1\leq j<n-1$, by induction, 
$$\ns(\widetilde{d}_{j,n-1})=\widetilde{u}_{n-1}^{q-2}\sigma(x_1x_2x_3^q\cdots x_{n-j}^{q^{n-j-2}}x_{n-j+1}^{q^{n-j}}\cdots x_{n-1}^{q^{n-2}}).$$
Since $\widetilde{u}_n=N(x_n)\widetilde{u}_{n-1}$, we have
$$\ns(\widetilde{d}_{1,n})=\widetilde{u}_n^{q-2}\sum_{j=1}^{n-1}\sigma(x_1x_2x_3^q \cdots x_{n-j}^{q^{n-j-2}}x_{n-j+1}^{q^{n-j}}\cdots x_{n-1}^{q^{n-2}})x_n^{q^{n-j-1}}.$$
Therefore $\ns(\widetilde{d}_{1,n})=\widetilde{u}_n^{q-2}\sigma(x_1x_2x_3^q \cdots \, x_n^{q^{n-2}})$.

For $i>1$, we have $\PP^{n-i}(\widetilde{d}_{i-1,n})= \widetilde{d}_{i,n}$. 
Since Steenrod operators take squares to squares, $\ns(\widetilde{d}_{i,n})=\ns(\PP^{n-i}(\ns(\widetilde{d}_{i-1,n})))$.
By induction 
$$\ns(\widetilde{d}_{i-1,n})=\widetilde{u}_n^{q-2}\sigma(x_1x_2x_3^q\cdots x_{n-i+2}^{q^{n-i}}x_{n-i+3}^{q^{n-i+2}}\cdots x_n^{q^{n-1}}).$$
Observe that 
 $$\PP^{n-i}(\sigma(x_1x_2x_3^q\cdots x_{n-i+2}^{q^{n-i}}x_{n-i+3}^{q^{n-i+2}}\cdots x_n^{q^{n-1}}))
 =\sigma(x_1x_2x_3^q\cdots x_{n-i+1}^{q^{n-i-1}}x_{n-i+2}^{q^{n-i+1}}\cdots x_n^{q^{n-1}}).$$
It then follows from Lemma~\ref{stn_un} that
$$\ns(\widetilde{d}_{i,n})=\widetilde{u}_n^{q-2}\sigma(x_1x_2x_3^q\cdots x_{n-i+1}^{q^{n-i-1}}x_{n-i+2}^{q^{n-i+1}}\cdots x_n^{q^{n-1}})$$
as required.
\end{proof}

\section{Symplectic Invariants}\label{sec: symplectics}

The ring of symplectic invariants, $S^{\syp{2m}{q}}$, is the complete intersection
generated by $\{\xi_1,\ldots, \xi_{2m}\}\cup\{d_{1,m},\ldots,d_{2m,2m}\}$
subject to the relations given in 
\cite[Theorem 8.3.11]{benson-polyinvafinigrou:93}.
For $i\leq 2m$, let $R_i$ denote the subalgebra of $S_m$ generated by 
$\{\xi_1,\ldots, \xi_i\}$. We will refer to a monomial $\xi_1^{j_1}\xi_2^{j_2}\cdots\xi_i^{j_i} \in R_i$
as a {\it natural monomial} if for every $k$ 
the base $q$ expansion of the exponent $j_k$ involves only $0$ and $1$. In a certain sense, these monomials are independent of $q$.
For a natural monomial $\beta$, we will call $\xi_j^{q^k}$ a {\it natural factor} of  $\beta$ if $\beta/\xi_j^{q^k}$ is a natural monomial.
We define the {\it natural degree} of $\beta$ to be the number of natural factors.

\begin{lemma} \label{nat_mon_lem} (i) $u_m$ is the sum of all natural monomials in $R_{2m-1}$ of degree $1+q+\cdots +q^{2m-1}$ and natural degree $m$.\\
(ii) $u_md_{i,m}$ is the sum all natural monomials in $R_{2m}$ of degree $1+q+\cdots +q^{2m-i-1}+q^{2m-i+1}+\cdots +q^{2m}$ and natural degree $m$.
\end{lemma}

\begin{proof} Recall that $\deg(u_m)=1+q+\cdots+q^{2m-1}$. It follows from \cite[Prop. 8.3.3]{benson-polyinvafinigrou:93}
that $u_m$ is the sum of natural monomials with $m$ distinct natural factors.
Using the matrix description of the Dickson invariants,
$u_m=\sigma(x_1x_2^q\cdots x_m^{q^{m-1}}y_m^{q^m}\cdots y_1^{q^{2m-1}})$.
Each term of $\sigma(x_1x_2^q\cdots x_m^{q^{m-1}}y_m^{q^m}\cdots y_1^{q^{2m-1}})$ appears in a unique natural monomial of degree 
$1+q+\cdots +q^{2m-1}$ and natural degree $m$. To see this,
for each term 
$x_1^{\alpha(1)}x_2^{\alpha(2)}\cdots x_m^{\alpha(m)}y_m^{\alpha(m+1)}\cdots y_1^{\alpha(2m)}$, we associate the partition of $\{1,q,\ldots,q^{2m-1}\}$
into subsets of size $2$ given by 
$$\{\{\alpha(1),\alpha(2m)\},\{\alpha(2),\alpha(2m-1)\}
\ldots,\{\alpha(m),\alpha(m+1)\}\}.$$
To each subset of size $2$, say $\{q^j,q^k\}$ with $j<k$, we associate the natural factor $\xi_{k-j}^{q^j}$. The term appears in the natural monomial given by the product of the resulting natural factors. Summing the natural monomials associated to the partitions gives $u_m$. Clearly these natural monomials have degree $1+q+\cdots +q^{2m-1}$ and natural degree $m$. Suppose $\beta$ is a natural monomial of degree
$1+q+\cdots +q^{2m-1}$ and compute $\deg(\beta)$
by summing the degrees of the natural factors base $q$. If this sum is performed without carries, then $\beta$ is associated to a partition. Otherwise the natural degree of $\beta$ is greater than $m$. This completes the proof of part (i).

The proof of (ii) is similar to the proof of (i).
From \cite[Prop. 8.3.3]{benson-polyinvafinigrou:93},
$u_md_{i,m}$ is the sum of natural monomials with $m$ distinct natural factors.
The matrix description of the Dickson invariants gives $u_md_{i,m}$
as an orbit sum of monomials. To each term in the orbit sum, we associate
a partition of $\{1,q,\ldots,q^{2m}\}\setminus\{q^{2m-i}\}$ into subsets of size $2$ and to each partition we associate a natural monomial of
degree $1+q+\cdots +q^{2m-i-1}+q^{2m-i+1}+\cdots +q^{2m}$ and natural degree $m$.
\end{proof}

\begin{example} 
$u_2=\xi_3\xi_1^q+\xi_2^{q+1}+\xi_1^{q^2+1}$ and $u_2d_{1,2}=\xi_4\xi_1^q+\xi_3^q\xi_2+\xi_2^{q^2}\xi_1$.
\end{example}

\begin{remark} \label{rem_u_m-1} By definition, $u_{m-1}$ is an element of $S_{m-1}$. Since $u_{m-1}\in R_{2m-3}$, we can use the inclusion of $R_{2m-3}$ into $S_m$ to interpret 
$u_{m-1}$ as an element of $S_m$. Using this interpretation and Lemma~\ref{nat_mon_lem}, $u_{m-1}$ is the sum of natural monomials of degree $1+q+\cdots +q^{2m-3}$.
Define $\overline{u_{m-1}d_1}\in R_{2m-2}$ by $\overline{u_{m-1}d_1}:=\PP^{q^{2m-3}}(u_{m-1})$.
Similarly, for $0<i<2m-2$, define  $$\overline{u_{m-1}d_{i+1}}:=\PP^{q^{2m-3-i}}(\overline{u_{m-1}d_i})\in R_{2m-2}.$$
Note that the embedding of $R_{2m-2}$ in $S_{m-1}$ takes  $\overline{u_{m-1}d_i}$ to $u_{m-1}d_{i,m-1}$
and, by Lemma~\ref{nat_mon_lem},  $\overline{u_{m-1}d_i}$ is the sum of  natural monomials in $R_{2m-2}$ of degree $1+q+q^2+\cdots +q^{2m-2}-q^{2m-2-i}$.
\end{remark}

Define a $2m\times (2m+1)$ matrix with entries in $R_{2m}$ by
$$M_m:=\begin{pmatrix}
      0& \xi_{1}& \xi_{2}      &\xi_3   &\xi_4&& \cdots && \xi_{2m}\\
\xi_1& 0         & \xi_{1}^q  & \xi_{2}^q   &  \xi_3^q&& \cdots && \xi_{2m-1}^q\\
\xi_2& \xi_1^q& 0            &\xi_1^{q^2} &\xi_2^{q^2} && \cdots && \xi_{2m-2}^{q^2}\\
&&&&\vdots&&&& \\
\xi_{m-1}&\xi_{m-2}^q&\cdots &\xi_1^{q^{m-2}}&0&\xi_1^{q^{m-1}}&\xi_2^{q^{m-1}}&\cdots &\xi_{m+1}^{q^{m-1}}\\
1&0&0&0&& \cdots &&0& P_{m,m}\\
0&1&0&0&&\cdots &0& P_{m-1,m-1}^q&P_{m-1,m}\\
&&&&\vdots&&&& \\
0&\cdots&0&1&0&P_{1,1}^{q^{m-1}}& P_{1,2}^{q^{m-2}}&\cdots&P_{1,m}
\end{pmatrix}
$$ 
where $P_{i,j}$ are defined as in \cite[Prop. 8.3.7]{benson-polyinvafinigrou:93}.
The matrix $M_m$ is the augmented coefficient matrix for the relations given in \cite[Theorem 8.3.11]{benson-polyinvafinigrou:93},
compare with the displayed matrix equation on page 96 of \cite{benson-polyinvafinigrou:93}. 
Let $M_m(j)$ denote the minor of $M_m$ formed by removing column $j$ from $M_m$.

Observe that removing row $1$, row $m$, column $1$ and column $2m+1$ from $M_m$ gives $\frob(M_{m-1})$, 
the matrix formed by taking the $q^{th}$ power of the entries of  $M_{m-1}$. 
Using this and computing $M_m(2m+1)$ by expanding first along row $m+1$ and then along row $1$ gives
\begin{equation}\label{expansion} M_m(2m+1)=\sum_{j=1}^{2m-1} \xi_j M_{m-1}(j)^q.\end{equation}
Lemma~\ref{nat_mon_lem} and Remark~\ref{rem_u_m-1} give
\begin{equation} \label{um-expr} u_m=\xi_{2m-1}u_{m-1}^q+\sum_{j=1}^{2m-2}\xi_j\overline{u_{m-1}d_{2m-j-1}}^q.
\end{equation}

\begin{theorem} \label{Oodd_det_thm} $u_m=M_m(2m+1)$ and $\overline{u_m d_i}=M_m(2m+1-i)$.
\end{theorem}
\begin{proof} The proof is by induction on $m$. For $m=1$ we have
$$M_1=\begin{pmatrix} 0&\xi_1&\xi_2\\ 1&0&P_{1,1}\end{pmatrix},$$
which gives $M_1(3)=\xi_1=u_1$, $M_1(2)=\xi_2=\overline{u_1d_1}$ and $M_1(1)=\xi_1P_{1,1}=\xi_1^q=\overline{u_1d_2}$,
as required.

For $m=2$, we have $u_2=\xi_3\xi_1^q+\xi_2^{q+1}+\xi_1^{q^2+1}$
and $M_2(5)=\xi_3\xi_1^q+\xi_2\xi_2^q+\xi_1P_{1,1}$. Since $P_{1,1}=\xi_1^{q-1}$, this gives $u_2=M_2(5)$.
The matrix form for the relations in $S_2^{\syp{4}{q}}$ is 
$$\begin{pmatrix} 0&\xi_1&\xi_2&\xi_3\\
\xi_1&0&\xi_1^q&\xi_2^q\\
1&0&0&0\\
0&1&0& P^q_{1,1}
\end{pmatrix}
\begin{pmatrix} d_{4,2} \\ d_{3,2} \\ d_{2,2} \\ d_{1,2} \end{pmatrix} =
\begin{pmatrix} \xi_4 \\ \xi_3^q \\ P_{2,2} \\ P_{1,2} \end{pmatrix} .
$$
Since we are in characteristic two, Cramer's rule gives
$d_{1,2}=M_2(4)/M_2(5)$, $d_{2,2}=M_2(3)/M_2(5)$, $d_{3,2}=M_2(2)/M_2(5)$, and $d_{4,2}=M_2(1)/M_2(5)$.
Note that these quotients are in $S_2$. Scaling by $u_2=M_2(5)$,  we get relations in $R_4\subset S_2$:
$\overline{u_2d_1}=M_2(4)$, $\overline{u_2d_2}=M_2(3)$, $\overline{u_2d_3}=M_2(2)$, $\overline{u_2d_4}=M_2(1)$.
This completes the proof for $m=2$.

Suppose $m>2$. Using the induction hypothesis, $u_{m-1}=M_{m-1}(2m-1)$ and $\overline{u_{m-1}d_i}=M_{m-1}(2m-1-i)$.
Substituting into Equation~\ref{um-expr} gives
$$u_m=\xi_{2m-1} M_{m-1}(2m-1)^q+\sum_{j=1}^{2m-2}\xi_j M_{m-1}(j)^q.$$
Using Equation~\ref{expansion} gives $u_m=M_m(2m+1)$.
It then follows from Cramer's rule that $\overline{u_m d_i}=M_m(2m+1-i)$.
 \end{proof}

In the following $\field$ is the algebraic closure of $\field_q$ and for an ideal $I=\langle f_1,\ldots,f_k\rangle\subset\field_q[V]$, we write
$\V(f_1,\ldots,f_k)$ for the variety in 
$\overline{V}:=V\otimes\field$ determined by $I$.
For $v\in \overline{V}$ we use $\frob(v)$ to denote the  Frobenius map on $v$. For $v\in\overline{V}$, we have
$$v=z(v)\omega+\sum_{i=1}^m(y_i(v)\basise_i+x_i(v)\basisf_i)$$
and
$$\frob(v)=(z(v))^q\omega+\sum_{i=1}^m\left((y_i(v))^q\basise_i+(x_i(v))^q\basisf_i\right).$$

\begin{theorem} \label{thm-var-sp} 
$\V(\xi_1,\ldots,\xi_{m})=
\cup\{g\V(y_1,\ldots,y_m)\mid g\in\orth{2m+1}{q}\}$.
\end{theorem}
\begin{proof} It is clear that 
$g\V(y_1,\ldots,y_m)\subset \V(\xi_1,\ldots,\xi_{m})$ for $g\in\orth{2m+1}{q}$. Suppose $v\in \V(\xi_1,\ldots,\xi_{m})$. We will show that
$gv\in \V(y_1,\ldots,y_m)$ for some $g\in\orth{2m+1}{q}$.

The proof is by induction on $m$. For $m=1$, we have 
$$\xi_1=x_1y_1(y_1^{q-1}+x_1^{q-1})=y_1\prod_{c\in\field_q}(x_1+cy_1).$$
It follows from Lemma~\ref{terminal_var} that $\orth{3}{q}$ acts 
on $\Span_{\field_q}\{y_1,x_1\}$ as $\syp{2}{q}=\SL{2}{q}$.
Therefore, if $\xi_1(v)=0$ and $y_1(v)\not=0$, there exists $g\in\orth{3}{q}$
such that $0=y_1g(v)=y_1(gv)$.

For $m>1$, define $\widetilde{v}:=v-y_m(v)\basise_m-x_m(v)\basisf_m$, let
$\widetilde{\xi_i}$ denote the restriction of $\xi_i$ to 
${\rm Span}_{\field}\{\basise_1,\ldots,\basise_{m-1},\omega,\basisf_{m-1},\ldots \basisf_1\}$  and
identify $\orth{2m-1}{q}$ with the point-wise stabiliser of  
${\rm Span}_{\field}\{\basise_m,\basisf_m\}$ in $\orth{2m+1}{q}$.

If $y_m(v)=0$ then $\widetilde{\xi_i}(\widetilde{v})=0$ for $i=1,\ldots,m$. By induction, there is an element $g\in\orth{2m-1}{q}<\orth{2m+1}{q}$
with $g\widetilde{v}\in 
{\rm Span}_{\field}\{\omega,\basisf_1,\ldots,\basisf_{m-1}\}$. 
Therefore $gv\in {\rm Span}_{\field}\{\omega,\basisf_1,\ldots,\basisf_m\}
=\V(y_1,\ldots,y_m)$.

Suppose then that $y_m(v)\not=0$.  
Note that $\V(\xi_1,\ldots,\xi_{m})$ is closed under scalar 
multiplication since the $\xi_i$ are homogeneous.  
Similarly each 
component $g\V(y_1,\ldots,y_m)$ is also closed under scalar
multiplication.  Hence we may scale $v$ so that $y_m(v)=1$.
Then we define $w:=v-\frob(v)$. Note that $y_m(w)=0$.
Since $$w=(z(v)-z(v)^q)\omega
+\sum_{j=1}^m\left( (y_j(v)-y_j(v)^q)\basise_j+(x_j(v)-x_j(v)^q)\basisf_j\right)$$
we have
$$\xi_i(w)=\sum_{j=1}^m (x_j(v)-x_j(v)^q)(y_j(v)-y_j(v)^q)^{q^i}
      +(x_j(v)-x_j(v)^q)^{q^i}(y_j(v)-y_j(v)^q).$$
This gives 
\begin{align*}
\xi_1(w)&=\xi_1(v)+\xi_{2}(v)+\xi_1(v)^q\quad
\text{and}\\
\xi_i(w)&=\xi_i(v)+\xi_{i+1}(v)+\xi_{i-1}(v)^q+\xi_i(v)^q
\quad\text{for }i>1.\\
\end{align*}
Therefore $\xi_i(w)=0$ for $i=1,\ldots,m-1$.
Since $y_m(w)=0$, we have $\widetilde{\xi_i}(\widetilde{w})=0$ for $i=1,\ldots, m-1$ and so by induction
there is an element $g\in\orth{2m-1}{q}<\orth{2m+1}{q}$
with $g\widetilde{w}\in 
{\rm Span}_{\field}\{\omega,\basisf_1,\ldots,\basisf_{m-1}\}$ and 
$$gw\in {\rm Span}_{\field}\{\omega,\basisf_1,\ldots,\basisf_m\}
=\V(y_1,\ldots,y_m).$$
Hence $(y_j-y_j^q)(gv)=y_j(gw)=0$ for $j=1,\ldots,m$ and thus $y_j(gv)\in\field_q$.
Since $g$ fixes ${\rm Span}_{\field}\{\basise_m,\basisf_m\}$, we have $y_m(gv)=1$.
For convenience, define $c_j=y_j(gv)\in\field_q$ for $j=1,\ldots,m-1$ and let $h$ denote the linear transformation
given by $zh=z$, $x_jh=x_j$ for $j=1,\ldots,m-1$, and 
$x_mh=x_m+\sum_{j=1}^{m-1} c_jx_j$, and 
$y_jh=x_j-c_jy_m$ for $j=1,\ldots,m-1$
and $y_mh=y_m$. 
Observe that $h\in\orth{2m+1}{q}$
and  $$(x_m^q-x_m)(hgv)=(x_m^q-x_m)(gv)+\sum_{j=1}^{m-1} c_j(x_j^q-x_j)(gv).$$
Since  $\xi_1(gv)=0$, using the definition of $c_j$, and putting $c_m=y_m(gv)=1$ 
we have $$
(x_m^q-x_m)(hgv)
=\sum_{j=1}^m c_j^q x_j^q (gv) - \sum_{j=1}^m c_j x_j (gv)
=\xi_1^q(gv) - \xi_1(gv) = 0.$$
Therefore $c:=x_m(hgv)\in\field_q$. 

Since $y_m(hgv)=1$, we have $(cy_m+x_m)(hgv)=0$.
Define $\alpha\in\orth{2m+1}{q}$ by $y_m\alpha=cy_m+x_m$,
$x_m\alpha=y_m$, $z\alpha=z+\sqrt{c}y_m$ and, for $j<m$, 
$y_j\alpha=y_j$ and $x_j\alpha=x_j$.
Then $y_m(\alpha h g v)=0$ and we can apply the induction argument as above.
\end {proof}

\section{Orthogonal Invariants}\label{orth_invs}

In this section we introduce the orthogonal invariants $e_i$.
For a monomial $\beta\in S_m$, we define the {\it support} of $\beta$ to be the number of hyperbolic pairs appearing in $\beta$,
i.e., the support of 
$\beta$ is $|\{i : x_i \text{ divides }\beta \text{ or } y_i \text{ divides }\beta\}|$.
The support of $\beta$ is at most $m$. Every term in $\xi_i$ is a monomial with support $1$.

We extend the definition of natural monomial to 
$R_{i}[\overline{\xi}_0]$ by also requiring the base $q$
expansion of the exponent on $\overline{\xi}_0$ to only 
involve $0$ and $1$.
Each term in $\sigma(x_1 x_2 x_3^q\cdots x_m^{q^{m-2}}y_m^{q^{m-1}}\cdots y_1^{q^{2m-2}})$ appears in a unique natural monomial in 
$R_{2m-2}[\overline{\xi}_0]$.
For example, when $m=2$, $x_1x_2y_2^qy_1^{q^2}$ is a term in $\xi_2\xi_1$ and $x_1y_1x_2^qy_2^{q^2}$ appears in $\overline{\xi}_0\xi_1^q$.
For these natural monomials, the terms with support $m$ correspond to terms in  $\sigma(x_1 x_2 x_3^q\cdots x_m^{q^{m-2}}y_m^{q^{m-1}}\cdots y_1^{q^{2m-2}})$.

\begin{definition}\label{delta_one_def} Let $\delta_{jk}$ denote the sum of the natural monomials of degree $q+q^2+\cdots +q^{2m-2}-q^j-q^k$ and natural degree $m-2$ in $R_{2m-2}$.
Define
$$\delta_1:=\sum_{0<j<k<2m-1}\xi_j\xi_k\delta_{jk}.$$
Note that for $m=1$, we have $\delta_1=0$.
\end{definition}

\begin{remark}\label{delta_one_rem} Arguing as in the proof of Lemma~\ref{nat_mon_lem}, $\delta_{jk}$ is the sum of the natural monomials associated to partitions
of $\{q,q^2,\ldots,q^{2m-2}\}\setminus\{q^j,q^k\}$ into subsets of size $2$. Furthermore,
it follows from part (ii) of Lemma~\ref{nat_mon_lem} and 
Remark~\ref{rem_u_m-1} that $\xi_j\delta_{jk}=\overline{u_{m-1}d_{2m-2-k}}$
and $\xi_k\delta_{jk}=\overline{u_{m-1}d_{2m-2-j}}$.
\end{remark}

\begin{lemma} \label{ns_sig_lem} 
$\sigma(x_1 x_2 x_3^q\cdots x_m^{q^{m-2}}y_m^{q^{m-1}}\cdots y_1^{q^{2m-2}})=\overline{\xi}_0u_{m-1}^q+\ns(\delta_1)$.
\end{lemma}
\begin{proof} For $m=1$, we have $\sigma(x_1y_1)=\overline{\xi}_0$
with $u_{m-1}=u_0=1$ and $\delta_1=0$. Suppose $m>1$.
Each term in $F:=\sigma(x_1 x_2 x_3^q\cdots x_m^{q^{m-2}}y_m^{q^{m-1}}\cdots y_1^{q^{2m-2}})$ 
appears in a unique natural monomial in $R_{2m-2}[\overline{\xi}_0]$.
To see this, note that term $$x_1^{\alpha(1)}x_2^{\alpha(2)}\cdots x_m^{\alpha(m)}y_m^{\alpha(m+1)}
\cdots y_1^{\alpha(2m)}$$ falls into one of two cases.
Either $\{\alpha(j),\alpha(2m-j+1)\}=\{1\}$ for some $j$
or $\{\alpha(j),\alpha(2m-j+1)\}\cap \{\alpha(k),\alpha(2m-k+1)\}=\{1\}$
for some $j<k$. In the first case, the associated natural monomial
is of the form $\overline{\xi}_0\beta^q$ where $\beta$ is a
natural monomial of degree $1+q+\ldots q^{2m-3}$ and natural degree $m-1$
(see part (i) of Lemma~\ref{nat_mon_lem}).
In the second case, the associated natural monomial is of the form 
$\xi_j\xi_k\beta$ where $\beta$ is a natural monomial of
of degree $q+q^2+\cdots +q^{2m-2}-q^j-q^k$ and natural degree $m-2$
(see Definition~\ref{delta_one_def} and Remark~\ref{delta_one_rem}).

Consider a natural monomial $\beta$ associated to one of the terms of $F$. The terms of $\beta$ with support $m$ are the terms appearing in $F$.
For the terms of $\beta$ with support $m-1$, if the factors associated to the duplicate hyperbolic pair are distinct, then the term appears in precisely two of the natural monomials.
Otherwise, we have terms like $x_1^2y_1^{q^2+q}\beta^q$. Thus terms of support $m-1$ don't contribute to $F=\ns(F)$.
For terms of support less than $m-1$, there are various cases but either the term is a square or appears in an even number of natural monomials. 
Therefore $F$ is the non-square part of 
$\overline{\xi}_0u_{m-1}^q+\delta_1$. 
Since $\ns(\overline{\xi}_0u_{m-1}^q)=\overline{\xi}_0u_{m-1}^q$,
the result follows.
\end{proof}

Define $e_1=u_m^{q/2-1}u_{m-1}^{q/2}z+b_1$  with $b_1$ determined by $$b_1^2=d_{1,m}+u_m^{q-2}(\delta_1+\overline{\xi}_0 u_{m-1}^q).$$

\begin{theorem} $e_1\in S[z]^{\orth{2m+1}{q}}$ and $\lt(e_1)=y_1^{q^{2m-1}(q-1)/2}$.
\end{theorem}
\begin{proof} First observe that $\lt(e_1)=y_1^{q^{2m-1}(q-1)/2}$ using either the lex or grevlex orders.
Taking $F=d_{1,m}+u_m^{q-2}\delta_1\in S^{\syp{2m}{q}}$, we have $\ns(F)=\ns(d_{1,m})+u_m^{q-2}\ns(\delta_1)$.
Using Lemmas~\ref{ns_di_lem} and \ref{ns_sig_lem} gives $\ns(F)=u_m^{q-2}u_{m-1}^q\overline{\xi}_0$.
Thus taking $b_1^2=F+\ns(F)=d_{1,m}+u_m^{q-2}\delta_1+u_m^{q-2}u_{m-1}^q\overline{\xi}_0$ and applying Lemma~\ref{inv-cons-lem}
gives an element $f=a_1z+b_1\in S[z]^{\orth{2m+1}{q}}$ with $a_1=u_m^{q/2-1}u_{m-1}^{q/2}$.
\end{proof}

For $0<i<2m-1$, define $e_{i+1}:=\PP^{q^{2m-1-i}/2}(e_i)$ and $\delta_{i+1}:=\PP^{q^{2m-1-i}}(\delta_i)$.

\begin{lemma} \label{ei-lem} For $1<i<2m$, $e_i=u_m^{q/2-1}\overline{u_{m-1}d_{i-1}}^{q/2}z+b_i$ with
$$b_i^2=d_{i,m}+u_m^{q-2}(\delta_i+\overline{\xi}_0\overline{u_{m-1}d_{i-1}}^q).$$
\end{lemma}
\begin{proof} Using Lemma~\ref{stn_un}, $\PP^k(u_m)=0$ for $0<k<q^{2m-1}$. 
Using the definition of $\overline{u_{m-1}d_{i-1}}$ (see Remark~\ref{rem_u_m-1}), $\PP^{q^{2m-i-1}}(\overline{u_{m-1}d_{i-2}})=\overline{u_{m-1}d_{i-1}}$.
Hence $e_i=a_iz+b_i$ where 
$a_i=\PP^{q^{2m-i}/2}(a_{i-1})=u_m^{q/2-1}\overline{u_{m-1}d_{i-1}}^{q/2}$ and $b_i=\PP^{q^{2m-i}/2}(b_{i-1})$.
 Furthermore,
using the action of the Steenrod algebra on the Dickson invariants and the definition of $\delta_i$, we have 
$b_i^2=d_{i,m}+u_m^{q-2}(\delta_i+\overline{\xi}_0\overline{u_{m-1}d_{i-1}}^q)$.
\end{proof}

\begin{definition} \label{delta_gen_def}For $0<i<2m$, $0<j<k<2m$ and $i\not\in\{2m-j,2m-k\}$, 
let $\delta_{jk}^{(i)}$ denote the sum of the natural monomials of
degree $q+q^2+\cdots +q^{2m-1}-q^{2m-i}-q^j-q^k$ and natural degree $m-2$
in $R_{2m-1}$.
Note that $\delta_{jk}^{(1)}=\delta_{jk}$.
\end{definition}

\begin{remark}\label{delta_gen_rem} Arguing as in the proof of Lemma~\ref{nat_mon_lem}, $\delta_{jk}^{(i)}$ is the sum of the natural monomials associated to partitions
of $\{q,q^2,\ldots,q^{2m-1}\}\setminus\{q^j,q^k,q^{2m-i}\}$ into subsets of size $2$. 
\end{remark}

\begin{lemma}\label{delta_gen_lem}
    For $q>2$ and $0<i<2m$,
    $$\delta_i=\sum\{\xi_j\xi_k\delta_{jk}^{(i)}\mid 
    0<j<k<2m,\, j\not=2m-i,\, k\not=2m-i\}.$$
\end{lemma}
\begin{proof} 
 The proof is by induction on $i$.
The case $i=1$ is Definition~\ref{delta_one_def}. 
By induction, we assume the result is true for $\delta_{i-1}$ with $i>1$.
Therefore $\delta_{i-1}$ is a sum of natural monomials with natural degree $m$. Suppose $\beta$ is one of these natural monomials.
It follows from Corollary~\ref{com_st} that $\PP^{q^{2m-i}}(\beta)$ is a
sum of terms each consisting of a product of $m$ natural factors.
Since $q>2$, the base $q$ digit sum of the degree of $\PP^{q^{2m-i}}(\beta)$
is $2m$. Therefore computing the degree by summing the degrees of $m$ natural factors is computed without carries. Thus the natural factors are distinct and $\delta_i=\PP^{q^{2m-i}}(\delta_{i-1})$ is a sum of natural monomials of natural degree $m$. Furthermore, using the base $q$ expansion of $\deg(\delta_i)$, each of these natural monomials is of the form $\xi_j\xi_k\alpha$ for
some $j$, $k$ and $\alpha$, where $\alpha$ is a natural monomial appearing in $\delta_{jk}^{(i)}$. To complete the proof, we need to show that each natural monomial of this form appears precisely once in
$\PP^{q^{2m-i}}(\delta_{i-1})$. 

 Recall that for a natural monomial $\beta$,  $\xi_\ell^{q^r}$ is a natural factor of  $\beta$ if $\beta/\xi_\ell^{q^r}$ is a natural monomial.
Note the degree of $\xi_\ell^{q^r}$ is $q^{r+\ell}+q^r$. We will refer to $q^{r+\ell}$ as the {\it head} of $\xi_\ell^{q^r}$ and $q^r$ as the {\it tail}. 
If $\beta$ is a natural monomial appearing in 
$\delta_{i-1}$, then $\beta$ has a natural factor, say $\xi_\ell^{q^r}$, such $q^{2m-i}$ is either the head or the tail of $\xi_\ell^{q^r}$. Write $\beta=\widetilde{\beta}\xi_\ell^{q^r}$.
It easy to see that $\widetilde{\beta}\PP^{q^{2m-i}}(\xi_\ell^{q^r})$ is a natural monomial of degree $2+q+\cdots +q^{2m-1}-q^{2m-i}$.
To prove the result, it is sufficient to show that $\PP^{q^{2m-i}}(\beta)=\widetilde{\beta}\PP^{q^{2m-i}}(\xi_\ell^{q^r})$.
We use the Cartan identity to distribute the action of $\PP^{q^{2m-i}}$ on $\beta$. 

Using the action of the Steenrod algebra on $R_{2m-1}$
(see Corollary~\ref{com_st}), $\PP^{q^{2m-i}}$ can only non-trivially distribute on a natural factor
if the tail of the natural factor is less than $q^{2m-i}$.
Since $q>2$, we have $2+q+\cdots +q^{2m-i-1}<q^{2m-i}$ and there are no additional distributions.

Alternatively, 
the Steenrod operation $\PP^{q^{2m-i}}$ fills the ``gap" of $q^{2m-i+1}$ in $\deg(\delta_{i-1})$ producing a ``gap" of $q^{2m-i}$ in $\deg(\delta_i)$.
Since both degree calculations are computed base $q$ without carries,
$\PP^{q^{2m-i}}$ must act on the natural factor contributing $q^{2m-i}$
to $\deg(\delta_{i-1})$.
\end{proof}

For a domain $A$, we use $\FF(A)$ to denote its field of fractions.

\begin{theorem} \label{thm_FF-char2} $\FF\left(S[z]^{\orth{2m+1}{q}}\right)=\field_q( \xi_1,\ldots,\xi_{2m} ,e_1)$.
\end{theorem}
\begin{proof} Using \cite[Thm 8.3.4]{benson-polyinvafinigrou:93}, we have $\FF\left(S^{\syp{}{}p{2m}{q}}\right)=\field_q(\xi_1,\ldots,\xi_{2m})$.
It follows from Lemma~\ref{terminal_var} and Campbell-Chuai \cite{campbell+chuai-local:07} that to compute $\FF\left(S[z]^{\orth{2m+1}{q}}\right)$,
it is sufficient to adjoin an invariant of degree one in $z$ to $\field_q(\xi_1,\ldots,\xi_{2m})$. One suitable choice is $e_1$.
\end{proof}

\begin{remark} Since $S^{\syp{2m}{q}}[u_m^{-1}]=\field_q[\xi_1,\ldots,\xi_{2m}][u_m^{-1}]$ (see \cite[Thm 8.3.4]{benson-polyinvafinigrou:93})
and $\lc_z(e_1)=u_m^{q/2-1}u_{m-1}^{q/2}$, we have 
$$S[z]^{\orth{2m+1}{q}}[u_m^{-1},u_{m-1}^{-1}]=\field_q[\xi_1,\ldots, \xi_{2m},e_1][u_m^{-1},u_{m-1}^{-1}].$$
\end{remark}

Define $\h:=\{\xi_0,\xi_1,\ldots,\xi_m,e_1,\ldots, e_m\}$.

\begin{theorem} \label{ortho_hsop_thm_char2} $\h$ is a homogeneous system of parameters.
\end{theorem}
\begin{proof} We will show that the variety  in $\overline{V}=\overline{\field_q}\otimes V$
cut out by the ideal generated by $\H$ is $\{\underline{0}\}$. 
Using Theorem~\ref{thm-var-sp},
$$\V( \xi_1,\ldots,\xi_{m})=\bigcup_{g\in \orth{2m+1}{q}}g\V(y_1,y_2,\ldots, y_m).$$
For $v\in \V(\xi_1,\xi_2,\ldots,\xi_m)$, choose $g\in \syp{2m}{q}$ so that $gv\in \V(y_1,y_2,\ldots,y_m)$. 
Since $e_i(gv)=e_i(v)$ and $\xi_0(gv)=\xi_0(v)$, to show that $\V(\H)=\{\underline{0}\}$, it is sufficient to show that
$\V(y_1,\ldots,y_m,\xi_0,e_1,\ldots,e_m)=\{\underline{0}\}$. To do this we work modulo the ideal $I:=\langle y_1,\ldots,y_m\rangle$.
Since $\xi_0\equiv_I z^2$ and $e_i \equiv_I b_i$, it is sufficient to show that $\V(y_1,\ldots,y_m,z,b_1,\ldots,b_m)=\{\underline{0}\}$.
Since $b_i^2\equiv_I d_{i,m}$, it is sufficient to show that $\V(y_1,\ldots,y_m,z,d_{1,m},\ldots,d_{m,m})=\{\underline{0}\}$.
Using the description of the $d_{i,m}$ as the coefficients of the polynomial $$\prod\{t+x_1g\mid g\in \gl{2m}{q}\},$$ we see that
$d_{i,m}\equiv_I(\widetilde{d}_{i,m})^{q^m}$. Since $\{\widetilde{d}_{1,m},\ldots,\widetilde{d}_{m,m}\}$ is a homogeneous system of parmeters for
$\gl{m}{q}$ acting on $\field_q[x_1,\ldots,x_m]$, $$\V(y_1,\ldots,y_m,z,d_{1,m},\ldots,d_{m,m})=\{\underline{0}\}$$ as required.
\end{proof}

\begin{example} In this example we consider the case $m=1$. Recall that the order of $\orth{3}{q}$ is $q(q^2-1)$.
Using Theorem~\ref{ortho_hsop_thm_char2}, $\{\xi_0,\xi_1,e_1\}$ is a homogeneous system of parameters.
Since $e_1=\xi_1^{q/2-1}z+b_1$, the product of the degrees of these three invariants is $2\cdot (q+1)\cdot(q^2-q)/2=q(q^2-1)$. Therefore 
$S_1[z]^{\orth{3}{q}}=\field_q[\xi_0,\xi_1,e_1]$. 
\end{example}

\begin{remark} The order of $\orth{5}{q}$ is $q^4(q^2-1)(q^4-1)$. 
The product of the degrees of the elements of $\h$ for $m=2$ is
$2\cdot (q+1)\cdot (q^2+1) \cdot q^3(q-1)/2\cdot q^2(q^2-1)/2=q^5(q^2-1)(q^4-1)/2$. The ratio of these two numbers is $q/2$.
Therefore, when $q=2$, $S[z]^{\orth{5}{2}}=\field_2[\xi_0,\xi_1,\xi_2,e_1,e_2]$ and for $q>2$, we expect $q/2$ module generators over 
$\field_q[\xi_0,\xi_1,\xi_2,e_1,e_2]$.
\end{remark}

\begin{example} In this example we consider the case $m=2$ and $q>2$. Let $A$ denote the subalgebra of $S_2[z]^{\orth{5}{q}}$ generated by
$\{\xi_0,\xi_1,\xi_2,\xi_3,e_1,e_2\}$. 
We have $e_1=u_2^{q/2-1}\xi_1^{q/2}z+b_1$ and $e_2=u_2^{q/2-1}\xi_2^{q/2}z+b_2$.
Furthermore, eliminating $z$ gives 
$e_1^2+u_2^{q-2}\xi_1^{q}\xi_0=d_{1,2}+u_2^{q-2}\delta_1$  and
$e_2^2+u_2^{q-2}\xi_2^q\xi_0=d_{2,2}+u_2^{q-2}\delta_2$.
Since $u_2$, $\delta_1$ and $\delta_2$ lie in $R_3$, we see that $d_{1,2},d_{2,2}\in A$. 
Since $S_2^{\syp{4}{q}}$ is generated by $\{\xi_1,\xi_2,\xi_3,d_{1,2},d_{2,2}\}$, we have
$S_2^{\syp{4}{q}}\subset A$. In particular, this shows that $\xi_4\in A$.
Therefore, using Theorem~\ref{thm_FF-char2}, $\FF(A)=\FF(S_2[z]^{\orth{5}{q}})$.
It follows from Theorem~\ref{ortho_hsop_thm_char2} that $A$ contains a homogeneous system of parameters.
Therefore, to show that $A=S_2[z]^{\orth{5}{q}}$, it is sufficient to show that $A$ is integrally closed in its field of fractions.

Define $F:=\xi_2^{q/2}e_1+\xi_1^{q/2}e_2$. 
Observe that $F=\xi_2^{q/2}b_1+\xi_1^{q/2}b_2\in S_2^{\syp{4}{q}}$.
Since $q>2$, 
$$\deg(F)=\frac{q}{2}(q^3+1)<(q-1)q^3=\deg(d_{1,2})<\deg(d_{2,2})<\deg(\xi_4).$$
Therefore $F\in R_3$ and $\xi_1^{q/2}e_2$ lies in the polynomial ring $\field_q[\xi_0,\xi_1,\xi_2,\xi_3,e_1]$.
Hence $A[\xi_1^{-1}]=\field_q[\xi_0,\xi_1,\xi_2,\xi_3,e_1][\xi_1^{-1}]$. 
Thus to show $A=S_2[z]^{\orth{5}{q}}$, it is sufficient to prove that $\xi_1$ is prime in $A$.

Using the  descriptions of $b_1^2$ and $b_2^2$ given above,
we have $$F^2=\xi_2^q b_1^2+\xi_1^q b_2^2=\xi_2^q d_{1,2}+\xi_1^q d_{2,2}+u_2^{ q-2}(\xi_2^q \delta_1 +\xi_1^q\delta_2).$$
Using \cite[Theorem 8.3.11]{benson-polyinvafinigrou:93}, we have 
$\xi_2^q d_{1,2}+\xi_1^q d_{2,2}=\xi_3^q+\xi_1u_2^{q-1}$.
Hence $F^2=\xi_3^q+u_2^{ q-2}(\xi_1u_2+\xi_2^q \delta_1 +\xi_1^q\delta_2)$.
Using Lemmas~\ref{ns_sig_lem} and \ref{ei-lem} with $m=2$ gives $\delta_1=\xi_1\xi_2$
and $\delta_2=\xi_1\xi_3$. Since $u_2=\xi_3\xi_1^q+\xi_2^{q+1}+\xi_1^{q^2+1}$,
we have $\xi_1u_2+\xi_2^q \delta_1 +\xi_1^q\delta_2=\xi_1^{q^2+2}$.
Therefore $F=\xi_3^{q/2}+u_2^{q/2-1}\xi_1^{q^2/2+1}$ giving the relation
$$\xi_3^{q/2}=\xi_2^{q/2}e_1+\xi_1^{q/2}e_2+u_2^{q/2-1}\xi_1^{q^2/2+1}.$$
Hence $\xi_3^{q/2}\equiv_{\langle \xi_1\rangle}\xi_2^{q/2}e_1$.
Since $t^{q/2}+\xi_2^{q/2}e_1$ is irreducible in $\field_q(\xi_0,\xi_2,e_1,e_2)$,
$A/\xi_1A$ embeds in the field $\field_q(\xi_0,\xi_2,e_1,e_2)/\langle t^{q/2}+\xi_2^{q/2}e_1\rangle$,
proving that $\xi_1$ is prime in $A$.
\end{example}

\section{\texorpdfstring{$\orth{7}{q}$}{Dimension 7}}\label{sec: dim7}
\label{dim7}

In this section we compute $S_3[z]^{\orth{7}{q}}$ for $q$ even with $q>2$.
Using Theorem~\ref{ortho_hsop_thm_char2}, $\H=\{\xi_0,\xi_1,\xi_2,\xi_3,e_1,e_2,e_3\}$ is a homogeneous system of parameters. 
Let $A$ denote the subalgebra of $S_3[z]^{\orth{7}{q}}$ generated by $\H\cup\{\xi_4,\xi_5,e_4,e_5\}$.
It follows from Lemma~\ref{ei-lem} and the definition of $e_1$ that $S_3^{\syp{6}{q}}\subset A$.
Thus $R_6\subset A$ and, using Theorem~\ref{thm_FF-char2}, $\FF(A)=\FF(S_3[z]^{\orth{7}{q}})$.
Therefore, to show that $A=S_3[z]^{\orth{7}{q}}$, it is sufficient to show that $A$ is integrally closed in its field of fractions.

Using Lemma~\ref{ei-lem} and the definition of $e_1$, we have 
$e_1=u_3^{q/2-1}u_2^{q/2}z+b_1$ and, for $i>1$,
$e_i=u_3^{q/2-1}\overline{u_2d_{i-1}}^{q/2}z+b_i$.
Define $$F_1:=\xi_4^{q/2}e_1+\xi_3^{q/2}e_2+\xi_2^{q/2}e_3+\xi_1^{q/2}e_4.$$
The coefficient of $z$ in $F_1$ is 
$u_3^{q/2-1}(\xi_4u_2+\xi_3\overline{u_2d_1}+\xi_2\overline{u_2d_2}+\xi_1\overline{u_2d_3})^{q/2}$.
The first relation for $S_2^{\syp{4}{q}}$ interpreted as a relation in $R_4$ gives
\begin{equation}\label{rel1-O7}
 \xi_4u_2+\xi_3\overline{u_2d_1}+\xi_2\overline{u_2d_2}+\xi_1\overline{u_2d_3}=0.
 \end{equation}
Therefore $F_1=\xi_4^{q/2}b_1+\xi_3^{q/2}b_2+\xi_2^{q/2}b_3+\xi_1^{q/2}b_4\in S_3^{\syp{6}{q}}$.
Using  the descriptions of the $b_i^2$ we have

\begin{eqnarray*}
F_1^2&=&\xi_4^q b_1^2+\xi_3^q b_2^2+\xi_2^q b_3^2+\xi_1^q b_4^2\\
&=&\xi_4^q d_{1,3} +\xi_3^q d_{2,3}+\xi_2^q d_{3,3}+\xi_1^q d_{4,3}\\
&&+ u_3^{q-2}\overline{\xi}_0(\xi_4u_2+\xi_3\overline{u_2d_1}+\xi_2\overline{u_2d_2}+\xi_1\overline{u_2d_3})^q\\
&&+u_3^{q-2}(\xi_4^q \delta_1+\xi_3^q \delta_2+\xi_2^q \delta_3+\xi_1^q \delta_4).
\end{eqnarray*}
The second relation for $S_3^{\syp{6}{q}}$ is $\xi_5^q=\xi_4^q d_{1,3} +\xi_3^q d_{2,3}+\xi_2^q d_{3,3}+\xi_1^q d_{4,3}+\xi_1 u_3^{q-1}$.
Using this relation and Equation~\ref{rel1-O7} gives
$$F_1^2=\xi_5^q +u_3^{q-2}(\xi_1 u_3+\xi_4^q \delta_1+\xi_3^q \delta_2+\xi_2^q \delta_3+\xi_1^q \delta_4).$$

\begin{lemma}\label{O7-del-rel1} $\xi_1 u_3+\xi_4^q \delta_1+\xi_3^q \delta_2+\xi_2^q \delta_3+\xi_1^q \delta_4=\xi_1^2 u_2^{q^2}$.
\end{lemma}
\begin{proof} Using Lemma~\ref{nat_mon_lem}, $\xi_1 u_3$ is $\xi_1^2 u_2^{q^2}$ plus the sum of the natural monomials of degree
$2+2q+q^2+q^3+q^4+q^5$ and natural degree $4$ with $\xi_1$ as a natural factor. Using Definition~\ref{delta_one_def} and Lemma~\ref{delta_gen_lem}
we have
\begin{eqnarray*}
\delta_1&=&\xi_1\xi_2\xi_1^{q^3}+\xi_1\xi_3\xi_2^{q^2}+\xi_1\xi_4\xi_1^{q^2}
    +\xi_2\xi_3\xi_3^q+\xi_2\xi_4\xi_2^q+\xi_3\xi_4\xi_1^q,\\
\delta_2&=&\xi_1\xi_2\xi_2^{q^3}+\xi_1\xi_3\xi_3^{q^2}+\xi_1\xi_5\xi_1^{q^2}
    +\xi_2\xi_3\xi_4^q+\xi_2\xi_5\xi_2^q+\xi_3\xi_5\xi_1^q,\\
\delta_3&=&\xi_1\xi_2\xi_1^{q^4}+\xi_1\xi_4\xi_3^{q^2}+\xi_1\xi_5\xi_2^{q^2}
    +\xi_2\xi_4\xi_4^q+\xi_2\xi_5\xi_3^q+\xi_4\xi_5\xi_1^q,\\
\delta_4&=&\xi_1\xi_3\xi_1^{q^4}+\xi_1\xi_4\xi_2^{q^3}+\xi_1\xi_5\xi_1^{q^3}
    +\xi_3\xi_4\xi_4^q+\xi_3\xi_5\xi_3^q+\xi_4\xi_5\xi_2^q
    \,{\rm and}\\
\delta_5&=&\xi_2\xi_3\xi_1^{q^4}+\xi_2\xi_4\xi_2^{q^3}+\xi_2\xi_5\xi_1^{q^3}
       +\xi_3\xi_4\xi_3^{q^2}+\xi_3\xi_5\xi_2^{q^2}+\xi_4\xi_5\xi_1^{q^2}    
    \end{eqnarray*}
which gives
\begin{eqnarray*}
    \delta_1\xi_4^q+\delta_2\xi_3^q+\delta_3\xi_2^q+\delta_4\xi_2^q
&=&\xi_1\xi_2\left(\xi_1^{q^3}\xi_4^q+\xi_2^{q^3}\xi_3^q+\xi_1^{q^4}\xi_2^q\right)\\
&&+\xi_1\xi_3\left(\xi_2^{q^2}\xi_4^q+\xi_3^{q^2}\xi_3^q+\xi_1^{q^4}\xi_1^q\right)\\
&&+\xi_1\xi_4\left(\xi_1^{q^2}\xi_4^q+\xi_3^{q^2}\xi_2^q+\xi_2^{q^2}\xi_1^q\right)\\
&&+\xi_1\xi_5\left(\xi_1^{q^2}\xi_3^q+\xi_2^{q^2}\xi_2^q+\xi_1^{q^3}\xi_1^q\right)\\
&=&\xi_1u_3+\xi_1^2u_2^{q^2}
\end{eqnarray*}
as required.
\end{proof}

Using the lemma we have $F_1^2=\xi_5^q+u_3^{q-2}\xi_1^2u_2^{q^2}$. Hence $F_1=\xi_5^{q/2}+u_3^{q/2-1}\xi_1u_2^{q^2/2}$
giving the relation
\begin{equation}\label{O7-eqn1}
\xi_5^{q/2}=\xi_4^{q/2}e_1+\xi_3^{q/2}e_2+\xi_2^{q/2}e_3+\xi_1^{q/2}e_4+u_3^{q/2-1}\xi_1u_2^{q^2/2}.
\end{equation}

Define $$F_2:=\xi_3^{q^2/2}e_1+\xi_2^{q^2/2}e_2+\xi_1^{q^2/2}e_3+\xi_1^{q/2}e_5.$$
The coefficient of $z$ in $F_2$ is 
$u_3^{q/2-1}(\xi_3^qu_2+\xi_2^q\overline{u_2d_1}+\xi_1^q\overline{u_2d_2}+\xi_1u_2^q)^{q/2}$.
The second relation for $S_2^{\syp{4}{q}}$ interpreted as a relation in $R_4$ gives
\begin{equation}\label{prel2-O7}
 \xi_3^qu_2+\xi_2^q\overline{u_2d_1}+\xi_1^q\overline{u_2d_2}+\xi_1u_2^q=0.
 \end{equation}
Therefore $F_2=\xi_3^{q^2/2}b_1+\xi_2^{q^2/2}b_2+\xi_1^{q^2/2}b_3+\xi_1^{q/2}b_5\in S_3^{\syp{6}{q}}$.
Using the descriptions of the $b_i^2$ we have

\begin{eqnarray*}
F_2^2&=&\xi_3^{q^2} b_1^2+\xi_2^{q^2} b_2^2+\xi_1^{q^2} b_3^2+\xi_1^q b_5^2\\
&=&\xi_3^{q^2} d_{1,3} +\xi_2^{q^2} d_{2,3}+\xi_1^{q^2} d_{3,3}+\xi_1^q d_{5,3}\\
&&+ u_3^{q-2}\overline{\xi}_0(\xi_3^q u_2+\xi_2^q\overline{u_2d_1}+\xi_1^q\overline{u_2d_2}+\xi_1u_2^q)^q\\
&&+u_3^{q-2}(\xi_3^{q^2} \delta_1+\xi_2^{q^2} \delta_2+\xi_1^{q^2} \delta_3+\xi_1^q \delta_5).
\end{eqnarray*}
The third relation for $S_3^{\syp{6}{q}}$ is $$\xi_4^{q^2}=\xi_3^{q^2} d_{1,3} +\xi_2^{q^2} d_{2,3}+\xi_1^{q^2} d_{3,3}+\xi_1^q d_{5,3}+\xi_2 u_3^{q-1}.$$
Using this relation and Equation~\ref{prel2-O7} gives
$$F_2^2=\xi_4^{q^2} +u_3^{q-2}(\xi_2 u_3+\xi_3^{q^2} \delta_1+\xi_2^{q^2} \delta_2+\xi_1^{q^2} \delta_3+\xi_1^q \delta_5).$$

\begin{lemma}\label{O7-del-rel2} $\xi_2 u_3+\xi_3^{q^2} \delta_1+\xi_2^{q^2} \delta_2+\xi_1^{q^2} \delta_3+\xi_1^q \delta_5=\xi_2^2 \overline{u_2d_3}^q$.
\end{lemma}
\begin{proof} Using Lemma~\ref{nat_mon_lem}, $\xi_2 u_3$ is $\xi_2^2 \overline{u_2d_3}^q$ plus the sum of the natural monomials of degree
$2+q+2q^2+q^3+q^4+q^5$ and natural degree $4$ with $\xi_2$ as a natural factor.  Using the descriptions of the $\delta_i$ 
from Lemma~\ref{O7-del-rel1} we have
\begin{eqnarray*}
\delta_1\xi_3^{q^2}+\delta_2\xi_2^{q^2}+\delta_3\xi_1^{q^2}+\delta_5\xi_1^q
&=&\xi_1\xi_2\left(\xi_1^{q^3}\xi_3^{q^2}+\xi_2^{q^3}\xi_2^{q^2}+\xi_1^{q^4}\xi_1^{q^2}\right)\\
&&+\xi_2\xi_3\left(\xi_2^{q^2}\xi_4^q+\xi_3^{q^2}\xi_3^q+\xi_1^{q^4}\xi_1^q\right)\\
&&+\xi_2\xi_4\left(\xi_1^{q^2}\xi_4^q+\xi_3^{q^2}\xi_2^q+\xi_2^{q^2}\xi_1^q\right)\\
&&+\xi_2\xi_5\left(\xi_1^{q^2}\xi_3^q+\xi_2^{q^2}\xi_2^q+\xi_1^{q^3}\xi_1^q\right)\\
&=&\xi_2u_3+\xi_2^2\overline{u_2d_3}^q
\end{eqnarray*}
as required.
\end{proof}

Using the Lemma we have $F_2^2=\xi_4^{q^2}+u_3^{q-2}\xi_2^2\overline{u_2d_3}^q$.
Hence $F_2=\xi_4^{q^2/2}+u_3^{q/2-1}\xi_2\overline{u_2d_3}^{q/2}$ giving the relation
\begin{equation}\label{O7-eqn2}
\xi_4^{q^2/2}=\xi_3^{q^2/2}e_1+\xi_2^{q^2/2}e_2+\xi_1^{q^2/2}e_3+\xi_1^{q/2}e_5+u_3^{q/2-1}\xi_2\overline{u_2d_3}^{q/2}.
\end{equation}

Define $P_5:=e_5+u_2^{q(q-1)/2}e_1$. The coefficient of $z$ in $e_5$ is $u_3^{q/2-1}\overline{u_2d_4}^{q/2}=u_3^{q/2-1}u_2^{q^2/2}$.
Therefore the coefficient of $z$ in $P_5$ is zero and $P_5\in S_3^{\syp{6}{q}}$. 
Since $$\deg(e_5)=(q^6-q)/2<(q-1)q^5=\deg(d_{1,3})<q^6+1=\deg(\xi_6),$$
we have $P_5\in R_5$. This gives the relation
\begin{equation}\label{O7-eqn3}
e_5=u_2^{q(q-1)/2}e_1+P_5\in R_5[e_1].
\end{equation}

Define $P_4:=e_4+\xi_1^{(q-1)q^2/2}e_2+P_{1,2}^{q/2}e_1$ where $P_{1,2}\in R_3$
is defined as in \cite[Prop. 8.3.7]{benson-polyinvafinigrou:93}.
The coefficient of $z$ in $P_4$ is $u_3^{q/2-1}(\overline{u_2d_3}+(\xi_1^{q-1})^q\overline{u_2d_1}+P_{1,2}u_2)^{q/2}$,
which is zero using the fourth relation in $S_2^{\syp{4}{q}}$. Therefore $P_4\in S_3^{\syp{6}{q}}$. 
Since $\deg(e_4)=(q^6-q^2)/2<\deg(d_{1,3})$, we have $P_4\in R_5$ . This gives the relation
\begin{equation}\label{O7-eqn4}
e_4=\xi_1^{(q-1)q^2/2}e_2+P_{1,2}^{q/2}e_1+P_4.
\end{equation}

Define
$$\widetilde{M}:=\begin{pmatrix}
0&\xi_1^{q/2}& \xi_2^{q/2} & \xi_3^{q/2}\\
\xi_1^{q/2}& 0& \xi_1^{q^2/2}&\xi_2^{q^2/2}\\
1&0&0&0\\
0&1&0& \xi_1^{(q-1)q^2/2}
\end{pmatrix}$$
and observe that $\det(\widetilde{M})=u_2^{q/2}$.
The relations can be written in matrix form as
$$\widetilde{M} \begin{pmatrix} e_5\\ e_4 \\ e_3 \\ e_2 \end{pmatrix} 
=\begin{pmatrix} \xi_4^{q/2}e_1+\xi_5^{q/2}+u_3^{q/2-1}\xi_1u_2^{q^2/2}\\ 
\xi_3^{q^2/2}e_1+\xi_4^{q^2/2}+u_3^{q/2-1}\xi_2\overline{u_2d_3}^{q/2} \\ 
u_2^{q(q-1)/2}e_1+P_5 \\
P_{1,2}^{q/2} e_1+P_4 \end{pmatrix} .$$
 Observe that the entries of $\widetilde{M}$ lie in $R_3$ and the right hand side lies in $R_5[e_1]$.
 After inverting $u_2$, we can solve for $e_2$, $e_3$, $e_4$ and $e_5$ in $R_5[e_1,u_2^{-1}]$.
 Thus $A[u_2^{-1}]=R_5[\xi_0,e_1,u_2^{-1}]$.
 Since $\{\xi_0,\xi_1,\ldots,\xi_5,e_1\}$ is algebraically independent, to show that $A=S_3[z]^{\orth{7}{q}}$,
 it is sufficient to prove that $u_2$ is prime in $A$.

\begin{lemma} \label{O7-alg-ind} The set $\{\xi_0,\xi_1,\xi_2,\xi_4,\xi_5,e_1,e_2\}$ is algebraically independent.
\end{lemma} 
\begin{proof} Let $L$ denote the field generated by $\{\xi_0,\xi_1,\xi_2,\xi_4,\xi_5,e_1,e_2\}$.
Let $K$ denote the field generated by $\{\xi_1,\xi_2,\xi_4,\xi_5,d_{1,3},d_{2,3}\}$. 
We will show that the transcendence  degree of $K$ is $6$.
Since $K(\xi_0)$ has transcendence  degree $1$ over $K$ and $L$ is a finite extension of $K(\xi_0)$,
this shows that $L$ has transcendence degree $7$, proving the result.

We will use the expressions for $u_3d_{1,3}$ and $u_3d_{2,3}$ as elements in $R_6$ given by Lemma~\ref{nat_mon_lem}.
Since the expression for $u_3d_{1,3}$ has degree $1$ as a polynomial in $\xi_6$, $K(\xi_3)=\field_q(\xi_1,\ldots,\xi_6)=\field_q(\xi_1,\ldots,\xi_5,d_{1,3})$
has transcendence degree $6$. We will show that $K\subset K(\xi_3)$ is a finite extension.

Cross multiplying to eliminate $\xi_6$, define $$F:=\overline{u_2 d_1}^q u_3 d_{1,3} -u_2^q u_3 d_{2,3}\in R_5.$$
Dividing by $u_3$ gives $\overline{u_2 d_1}^q d_{1,3} -u_2^q d_{2,3}=F/u_3\in R_5$ (see \cite[Lemma 8.3.5]{benson-polyinvafinigrou:93}).
Recall that $u_2=\xi_3\xi_1^q+\xi_2^{q+1}+\xi_1^{q^2+1}$ and $u_2d_{1,2}=\xi_4\xi_1^q+\xi_3^q\xi_2+\xi_2^{q^2}\xi_1$.
Define $\overline{F}(t)\in \field_q[\xi_1,\xi_2,\xi_4,\xi_5][t]$ so that $\overline{F}(\xi_3)=F/u_3$.
Define
$$H(t):=(\xi_4\xi_1^q+t^q\xi_2+\xi_2^{q^2}\xi_1)^qd_{1,3}+(t\xi_1^q+\xi_2^{q+1}+\xi_1^{q^2+1})^qd_{2,3}+\overline{F}(t)\in K[t]$$
and observe that $H(\xi_3)=0$. Therefore, as long as $H$ is not identically zero, $\xi_3$ is a root of a polynomial in $K[t]$ and
the field extension $K\subset K(\xi_3)$ is finite. To see that $H$ is not identically zero, note that the coefficient of $t^{q^2}$
is $\xi_2^q d_{1,3}+\varepsilon$ for some $\varepsilon\in \field_q[\xi_1,\xi_2,\xi_4,\xi_5]$ and
$\{\xi_1,\xi_2,\xi_4,\xi_5,d_{1,3}\}$ is algebraically independent.
\end{proof}

\begin{theorem}\label{O7-main-thm} (a) $A=S_3[z]^{\orth{7}{q}}$.\\
(b) $S_3[z]^{\orth{7}{q}}$ is a complete intersection with relations given by Equations~\ref{O7-eqn1}, \ref{O7-eqn2}, \ref{O7-eqn3} and \ref{O7-eqn4}.\\
(c) $S_3[z]^{\orth{7}{q}}$ is a free module over $\field_q[\H]$ with module generators given by the monomial factors of $\xi_5^{q/2-1}\xi_4^{q^2/2-1}$.
\end{theorem}
\begin{proof} We denote by $T$ the polynomial algebra $R_5[\xi_0,E_1,E_2,E_3,E_4,E_5]$ and let $\rho:T\to A$ denote the algebra epimorphism taking $E_i$ to $e_i$.
Define 
\begin{eqnarray*}
r_1&:=&\xi_5^{q/2}+\xi_4^{q/2}E_1+\xi_3^{q/2}E_2+\xi_2^{q/2}E_3+\xi_1^{q/2}E_4+u_3^{q/2-1}\xi_1u_2^{q^2/2},\\
r_2&:=&\xi_4^{q^2/2}+\xi_3^{q^2/2}E_1+\xi_2^{q^2/2}E_2+\xi_1^{q^2/2}E_3+\xi_1^{q/2}E_5+u_3^{q/2-1}\xi_2\overline{u_2d_3}^{q/2},\\
r_3&:=&E_5+u_2^{q(q-1)/2}E_1+P_5 \, {\rm and} \\
r_4&:=&E_4+\xi_1^{(q-1)q^2/2}E_2+P_{1,2}^{q/2}E_1+P_4.
\end{eqnarray*}
Let $\overline{T}$ denote the quotient $T/\langle r_1,r_2,r_3,r_4\rangle$ and 
observe that $\rho$ induces an epimorphism from $\overline{T}$ to $A$, say $\overline{\rho}$.

Note that 
\begin{eqnarray*}
r_1&\equiv& \xi_5^{q/2} \mod{\langle \xi_1,\xi_2,\xi_3,\xi_4\rangle},\\ 
r_2&\equiv& \xi_4^{q^2/2} \mod{\langle \xi_1,\xi_2,\xi_3\rangle},\\
r_3&\equiv& E_5 \mod{\langle \xi_1,\xi_2,\xi_3,\xi_4,\xi_5\rangle} \, {\rm and} \\
r_4&\equiv &E_4 \mod{\langle \xi_1,\xi_2,\xi_3,\xi_4,\xi_5\rangle}.
\end{eqnarray*}
Since $u_2=\xi_3\xi_1^q+\xi_2^{q+1}+\xi_1^{q^2+1}$, we also have $u_2\equiv_{\langle\xi_1\rangle}\xi_2^{q+1}$.
Therefore $$\xi_0,\xi_1,u_2,\xi_3, r_2, r_1,r_3,r_4,E_1,E_2,E_3$$ is a regular sequence in $T$.
In a graded polynomial ring, a partial homogeneous system of parameters is regular sequence. 
Therefore $r_1,r_2,r_3,r_4$ is a regular sequence  in $T$ and $\overline{T}$ is a complete intersection.
Furthermore, the equivalence classes of $\xi_0,\xi_1,u_2,\xi_3,E_1,E_2,E_3$ are a regular sequence in $\overline{T}$.
Since $\overline{T}$ has Krull dimension $7$, this means that $\overline{T}$ is Cohen-Macauley. 

In  the following, we identify $R_5$ with its image in $\overline{T}$ and use $\overline{E_i}$ to denote the equivalence class of $E_i$ in $\overline{T}$.
In $\overline{T}$, we have
$$\widetilde{M} \begin{pmatrix} \overline{E_5}\\ \overline{E_4} \\ \overline{E_3} \\ \overline{E_2} \end{pmatrix} 
=\begin{pmatrix} \xi_4^{q/2}\overline{E_1}+\xi_5^{q/2}+u_3^{q/2-1}\xi_1u_2^{q^2/2}\\ 
\xi_3^{q^2/2}\overline{E_1}+\xi_4^{q^2/2}+u_3^{q/2-1}\xi_2\overline{u_2d_3}^{q/2} \\ 
u_2^{q(q-1)/2}\overline{E_1}+P_5 \\
P_{1,2}^{q/2} \overline{E_1}+P_4 \end{pmatrix} .$$
Since $\det(\widetilde{M})=u_2^{q/2}$, after inverting $u_2$, we can eliminate $\overline{E_2}$, $\overline{E_3}$, $\overline{E_4}$ and $\overline{E_5}$.
This gives $\overline{T}[u_2^{-1}]=R_5[\xi_0,\overline{E_1}][u_2^{-1}]$. Note that, since $\FF(A)=\field_q(\xi_0,\xi_1,\ldots,\xi_5,e_1)$,
the set $\{\xi_0,\xi_1,\ldots,\xi_5,\overline{E_1}\}$ is algebraically independent, which means that $R_5[\xi_0,\overline{E_1}]$ is a UFD.

We know that $\xi_1$, $u_2$ is a regular sequence in $\overline{T}$ and $\xi_3\xi_1^q\equiv_{\langle u_2\rangle} \xi_2^{q+1}+\xi_1^{q^2+1}$.
We can use this congruence to eliminate $\xi_3$ in $\overline{T}[\xi_1^{-1}]/\langle u_2\rangle$.
Similarly, we can use $r_2$ to eliminate $\overline{E_3}$ in $\overline{T}[\xi_1^{-1}]/\langle u_2\rangle$.
Using $r_4$ and $r_5$ we can eliminate $\overline{E_5}$ and $\overline{E_4}$. 
This gives a correspondence between elements of $\overline{T}[\xi_1^{-1}]/\langle u_2\rangle$ and
elements of $\field_q[\xi_0,\xi_1,\xi_2,\xi_4,\xi_5,\overline{E_1},\overline{E_2}][\xi_1^{-1}]$.
From Lemma~\ref{O7-alg-ind}, $\{\xi_0,\xi_1,\xi_2,\xi_4,\xi_5,e_1,e_2\}$ is an algebraically independent subset of $A$,
which means that $u_2$ is prime in $\overline{T}[\xi_1^{-1}]$ and, therefore, prime in $\overline{T}$ (see \cite[Proposition 1.1]{Kropholler+Rajaei+Segal:05}).
This proves that $\overline{T}$ is integrally closed in its field of fractions.
Since $\overline{\rho}$ induces an isomorphism on fraction fields, it is injective.
Thus $\overline{T}$ is isomorphic to $A$, proving that $A$ is integrally closed in its field of fractions and, therefore,
 $A=S_3[z]^{\orth{7}{q}}$. Furthermore, since $\overline{T}$ is a complete intersection, $S_3[z]^{\orth{7}{q}}$ is a complete intersection,
 completing the proof of parts (a) and (b). Since $S_3[z]^{\orth{7}{q}}$ is a complete intersection, it is Cohen-Macaulay. Thus 
 $S_3[z]^{\orth{7}{q}}$  is a free module over $\field_q[\H]$ of rank
 $$\frac{2(q+1)(q^2+1)(q^3+1)(q^5(q-1)/2)(q^4(q-1)/2)(q^3(q-1)/2)}{q^9(q^2-1)(q^4-1)(q^6-1)}=q^3/4.$$
 Since the monomial factors of $\xi_5^{q/2-1}\xi_4^{q^2/2-1}$ are a spanning set of size $q^3/4$, they form a basis, proving part (c).
\end{proof}

\begin{remark}\label{O7-rem}
    When $q=2$ and $m=3$, we have
    $$\delta_1=\xi_1\xi_2\xi_1^8+\xi_1\xi_3\xi_2^4+\xi_1\xi_4\xi_1^4
    +\xi_2\xi_3\xi_3^2+\xi_2\xi_4\xi_2^2+\xi_3\xi_4\xi_1^2.$$
    Applying Steenrod operations gives 
    \begin{eqnarray*}
        \delta_2&=&\xi_1\xi_2\xi_2^8+\xi_1\xi_3\xi_3^4+\xi_1\xi_5\xi_1^4
    +\xi_2\xi_3\xi_4^2+\xi_2\xi_5\xi_2^2+\xi_3\xi_5\xi_1^2+\xi_1^4\xi_3^4,\\
    \delta_3&=&\xi_1\xi_2\xi_1^{16}+\xi_1\xi_4\xi_3^4+\xi_1\xi_5\xi_2^4
    +\xi_2\xi_4\xi_4^2+\xi_2\xi_5\xi_3^2+\xi_4\xi_5\xi_1^2,\\
    \delta_4&=&\xi_1\xi_3\xi_1^{16}+\xi_1\xi_4\xi_2^8+\xi_1\xi_5\xi_1^8
    +\xi_3\xi_4\xi_4^2+\xi_3\xi_5\xi_3^2+\xi_4\xi_5\xi_2^2+\xi_1^{20}
    \,{\rm and}\\
    \delta_5&=&\xi_2\xi_3\xi_1^{16}+\xi_2\xi_4\xi_2^8+\xi_2\xi_5\xi_1^8
       +\xi_3\xi_4\xi_3^4+\xi_3\xi_5\xi_2^4+\xi_4\xi_5\xi_1^4
    +\xi_2^2\xi_3^2\xi_4^2.    
    \end{eqnarray*}
    Using Lemma~\ref{nat_mon_lem} and these expressions for the $\delta_i$,
    we get
    \begin{eqnarray*}
        \delta_1\xi_4^2+\delta_2\xi_3^2+\delta_3\xi_2^2+\delta_4\xi_2^2
        &=&\xi_1u_3+\xi_1^2 u_2^4+\xi_1^{22}+\xi_1^4\xi_3^6 \,{\rm and}\\
        \delta_1\xi_3^4+\delta_2\xi_2^4+\delta_3\xi_1^4+\delta_5\xi_1^2
        &=&\xi_2u_3+\xi_2^2\overline{u_2d_3}^2
        +\xi_1^4\xi_2^4\xi_3^4+\xi_1^2\xi_2^2\xi_3^2\xi_4^2
    \end{eqnarray*}
    (compare with Lemmas \ref{O7-del-rel1} and \ref{O7-del-rel2})
    which give the relations
    \begin{eqnarray*}
        \xi_5&=&\xi_4e_1+\xi_3e_2+\xi_2e_3+\xi_1e_4+\xi_1 u_2^2
        +\xi_1^{11}+\xi_1^2\xi_3^3 \, {\rm and}\\
        \xi_4^2&=&\xi_3^2e_1+\xi_2^2e_2+\xi_1^2e_3+\xi_1e_5+\xi_1 u_2^2
        +\xi_1^2\xi_2^2\xi_3^2+\xi_1\xi_2\xi_3\xi_4
    \end{eqnarray*}
    (compare with Equations \ref{O7-eqn1} and \ref{O7-eqn2}).
    Observe that Equation~\ref{O7-eqn3}, Equation~\ref{O7-eqn4}
    and Lemma~\ref{O7-alg-ind} are valid for $q=2$.
    Arguing as in the proof of Theorem~\ref{O7-main-thm},
    we conclude that $S_3[z]^{\orth{7}{2}}$ is the hypersurface
    generated by $\{\xi_0,\xi_1,\xi_2,\xi_3,\xi_4,e_1,e_2,e_3\}$
    subject to a relation which rewrites $\xi_4^2$.
    This is consistent with \cite[Theorem 6.1]{Kropholler+Rajaei+Segal:05}.
    \end{remark}

\section{Relations}\label{sec: relations}

The following generalises Lemmas \ref{O7-del-rel1} and \ref{O7-del-rel2}.
\begin{lemma} \label{del_rel} Suppose $q>2$.\\
(a) $\xi_1u_m+\xi_1^2u_{m-1}^{q^2}=\xi_{2m-2}^q\delta_1+\xi_{2m-3}^q\delta_2+\cdots+\xi_1^q\delta_{2m-2}$.\\
(b) For $m>i>1$,
$$\xi_iu_m+\xi_i^2\overline{u_{m-1}d_{2m-1-i}}^q=\sum_{j=1}^{2m-1-i}\delta_j\xi_{2m-i-j}^{q^i}+\sum_{k=1}^{i-1}\xi_k^{q^{i-k}}\delta_{2m-i+k}.$$
\end{lemma}

\begin{proof} For $q>2$, Definition~\ref{delta_one_def} and
Lemma~\ref{delta_gen_lem} give $\delta_j$ as a sum of natural monomials
of natural degree $m$. Note that $\xi_{2m-i-j}^{q^i}$ has tail $q^i$
and head $q^{2m-j}$ while $\xi_k^{q^{i-k}}$ has head $q^i$ and tail $q^{i-k}$.

Using Lemma~\ref{nat_mon_lem}, $\xi_1 u_m$ is $\xi_1^2 u_{m-1}^{q^2}$ plus the sum of the natural monomials of degree
$2+2q+q^2+q^3+\cdots +q^{2m-1}$ and natural degree $m+1$ with $\xi_1$ as a natural factor.
Since $q>2$, each natural monomial of degree $2+2q+q^2+q^3+\cdots +q^{2m-1}$ 
and natural degree $m+1$ has two natural factors with tail equal to $1$.
If one of these natural factors is $\xi_1$ then the natural monomial has one natural factor with tail $q$, say $\xi_j^q$.
In this case the natural monomial appears in $\xi_1 u_m$ and $\xi_j^q\delta_{2m-1-j}$. Otherwise, the natural monomial has
two natural factors with tail $q$ and appears twice in $\xi_{2m-2}^q\delta_1+\xi_{2m-3}^q\delta_2+\cdots+\xi_1^q\delta_{2m-2}$.
This completes the proof of part (a).

Suppose $m>i>1$. Using Lemma~\ref{nat_mon_lem}, $\xi_i u_m$ is $\xi_i^2 \overline{u_{m-1}d_{2m-1-i}}^q$ 
plus the sum of the natural monomials of degree
$2+q+q^2+q^3+\cdots +q^{2m-1}+q^i$ and natural degree $m+1$ with $\xi_i$ 
 as a natural factor. 
Since $q>2$, each natural monomial of degree $2+q+q^2+q^3+\cdots +q^{2m-1}+q^i$ and natural degree $m+1$ 
has two natural factors with tail equal to $1$.
If one of these natural factors is $\xi_i$ then there is either a natural factor with tail $q^i$ or a natural factor with head $q^i$.
In the first case the natural monomial appears in $\delta_j\xi_{2m-i-j}^{q^i}$ for some $j$ and in the second case
the natural monomial appears in $\xi_k^{q^{i-k}}\delta_{2m-i+k}$ for some $k$.
In either case, the monomial appears once on the right hand side of the expression and once on the left hand side.
If $\xi_i$ is not a natural factor then $q^i$ appears either twice as a tail, twice as a head, or once as a tail and once as a head.
In all three cases, the monomial appears twice in the right hand side of the expression. If $q^i$ appears twice as a tail then the natural monomial
appears in $\xi_{\ell}\xi_r\delta_{\ell r}^{(j)}\xi_{2m-i-j}^{q^i}$,
with $\ell<r$, $i<2m-j$ and $\{\ell,r\}\cap\{i,2m-j\}=\emptyset$,
for two choices for $j$.
If $q^i$ appears twice as a head then the natural monomial
appears in $\xi_k^{q^{i-k}}\xi_{\ell}\xi_r\delta_{\ell r}^{(2m-i+k)}$,
with $\ell<r$, $k<i$ and $\{\ell,r\}\cap\{i,i-k\}=\emptyset$,
for two choices for $k$. If $q^i$ appears once as a head and once as a tail,
the natural monomial appears in 
$\xi_{\ell}\xi_r\delta_{\ell r}^{(j)}\xi_{2m-i-j}^{q^i}$
for one choice of $j$ and in
$\xi_k^{q^{i-k}}\xi_{\ell}\xi_r\delta_{\ell r}^{(2m-i+k)}$
for one choice of $k$ (with $\ell<r$, $k<i<2m-j$
and $\{\ell,r\}\cap\{i,i-k,2m-j\}=\emptyset$).
This completes the proof of part (b).
\end{proof}

Define $F_1:=\xi_{2m-2}^{q/2} e_1+\xi_{2m-3}^{q/2} e_2+\cdots  +\xi_{1}^{q/2} e_{2m-2}$.
Recall that $e_1=u_m^{q/2-1}u_{m-1}^{q/2}z+b_1$ and, for $i>1$,
$e_i=u_m^{q/2-1}\overline{u_{m-1}d_{i-1}}^{q/2}z+b_i$ (see Lemma~\ref{ei-lem}). 
The coefficient of $z$ in $F_1$ is
$$u_m^{q/2-1}\left(u_{m-1}\xi_{2m-2}+\overline{u_{m-1}d_1}\xi_{2m-3}+\cdots +\overline{u_{m-1}d_{2m-3}}\xi_1\right)^{q/2}.$$
The first relation for $S_{m-1}^{\syp{2m-2}{q}}$ is
$$\xi_{2m-2}=d_{1,m-1}\xi_{2m-3}+d_{2,m-1}\xi_{2m-4}+\cdots +d_{2m-3,m-1}\xi_1.$$
Multiplying this by $u_{m-1}$ and interpreting the result as a relation in $R_{2m-2}$, we see that the coefficient of $z$ in $F_1$ is zero. 
Hence $F_1\in S_{m}^{\syp{2m}{q}}$ and
$$F_1^2=\xi_{2m-2}^{q} b^2_1+\xi_{2m-3}^{q} b^2_2+\cdots  +\xi_{1}^{q} b^2_{2m-2}.$$
Recall that $b_1^2=d_{1,m}+u_m^{q-2}(\delta_1+\overline{\xi}_0u_{m-1}^q)$ and, for $i>1$,\
$b_i^2=d_{i,m}+u_m^{q-2}(\delta_i+\overline{\xi}_0\overline{u_{m-1}d_i}^q)$ (see Lemma~\ref{ei-lem}).
Note that the coefficient of $\overline{\xi}_0$ in $F_1^2$ is the square of the coefficient of $z$ in $F_1$.
Therefore $$F_1^2=\xi_{2m-2}^{q} d_{1,m}+\cdots  +\xi_{1}^{q} d_{2m-2,m}
+u_m^{q-2}(\xi_{2m-2}^{q} \delta_1+\cdots  +\xi_{1}^{q} \delta_{2m-2}).$$
The second relation for $S_{m}^{\syp{2m}{q}}$ is
$$\xi^q_{2m-1}=\xi_{2m-2}^{q} d_{1,m}+\cdots  +\xi_{1}^{q} d_{2m-2,m}+\xi_1d_{2m,2m}.$$
Using this and part (a) of Lemma~\ref{del_rel} gives
$$F_1^2=\xi_{2m-1}^q+\xi_1 d_{2m,2m}+u_m^{q-2}(\xi_1u_m+\xi_1^2u_{m-1}^{q^2})$$
for $q>2$.
Since $d_{2m,2m}=u_m^{q-1}$, we are left with $F_1^2=\xi_{2m-1}^q+\xi_1^2u_{m-1}^{q^2}u_m^{q-2}$.
Therefore $F_1=\xi_{2m-1}^{q/2}+\xi_1u_{m-1}^{q^2/2}u_m^{q/2-1}$. This gives the relation

\begin{equation}\label{Oodd-eqn1}
\xi_{2m-1}^{q/2}=\xi_{2m-2}^{q/2}e_1+\xi_{2m-3}^{q/2}e_2+\cdots +\xi_{1}^{q/2}e_{2m-2}
+\xi_1u_{m-1}^{q^2/2}u_m^{q/2-1}.
\end{equation}

For $1<i<m$,
$$F_i:=\sum_{j=1}^{2m-1-i} \xi_{2m-i-j}^{q^i/2} e_j +\sum_{k=1}^{i-1} \xi_k^{q^{i-k}/2} e_{2m-i+k}.$$
Multiply the $i^{\rm th}$ relation for  $S_{m-1}^{\syp{2m-2}{q}}$ (see \cite[Theorem 8.3.11]{benson-polyinvafinigrou:93})
by $u_{m-1}$ and interpreting the result as a relation in $R_{2m-2}$,
we see that the coefficient of $z$ in $F_i$ is zero. Therefore $F_i\in S_{m}^{\syp{2m}{q}}$ and
$$F_i^2=\sum_{j=1}^{2m-1-i} \xi_{2m-i-j}^{q^i} b^2_j +\sum_{k=1}^{i-1} \xi_k^{q^{i-k}} b^2_{2m-i+k}.$$
Note that the coefficient of $\overline{\xi}_0$ in $F_i^2$ is the square of the coefficient of $z$ in $F_i$.
Hence 
\begin{eqnarray*} F_i^2&=&\sum_{j=1}^{2m-1-i} \xi_{2m-i-j}^{q^i} d_{j,m} +\sum_{k=1}^{i-1} \xi_k^{q^{i-k}} d_{2m-i+k,m}\\
&&+u_m^{q-2}\left(\sum_{j=1}^{2m-1-i} \xi_{2m-i-j}^{q^i} \delta_j +\sum_{k=1}^{i-1} \xi_k^{q^{i-k}} \delta_{2m-i+k}\right).
\end{eqnarray*}
Using the $i+1$ relation for $S_{m}^{\syp{2m}{q}}$ and part (b) of Lemma~\ref{del_rel} gives
$$F_i^2=\xi_{2m-i}^{q^i}+\xi_i d_{2m,2m}+u_m^{q-2}\left(\xi_i u_m+\xi_i^2\overline{u_{m-1}d_{2m-1-i}}^q\right)$$
for $q>2$.
Since $d_{2m,2m}=u_m^{q-1}$, we are left with 
$F_i^2=\xi_{2m-i}^{q^i}+\xi_i^2\overline{u_{m-1}d_{2m-1-i}}^{q}u_m^{q-2}$,
which gives $F_i=\xi_{2m-i}^{q^i/2}+\xi_i \overline{u_{m-1}d_{2m-1-i}}^{q/2} u_m^{q/2-1}$. 
This gives the relation
\begin{equation}\label{Oodd-eqni}
\xi_{2m-i}^{q^i/2}=\sum_{j=1}^{2m-1-i} \xi_{2m-i-j}^{q^i/2} e_j +\sum_{k=1}^{i-1} \xi_k^{q^{i-k}/2} e_{2m-i+k}
+\xi_i \overline{u_{m-1}d_{2m-1-i}}^{q/2} u_m^{q/2-1}.
\end{equation}

Recall that $e_{2m-1}=u_m^{q/2-1}\overline{u_{m-1}d_{2m-2}}^{q/2}z+b_{2m-1}$ and
$e_1=u_m^{q/2-1}u_{m-1}^{q/2}z+b_1$. Since $\overline{u_{m-1}d_{2m-2}}=u_{m-1}^q$,
$$P_{2m-1}:=e_{2m-1}-u_{m-1}^{(q^2-q)/2}e_1\in S_{m}^{\syp{2m}{q}}.$$
Furthermore
$$\deg(e_{2m-1})=(q^{2m}-q)/2<(q-1)q^{2m-1}=\deg(d_{1,m})<q^{2m}+1=\deg(\xi_{2m}).$$
Therefore $P_{2m-1}\in R_{2m-1}$ and
\begin{equation} \label{Oodd-exrel1}
e_{2m-1}=u_{m-1}^{(q^2-q)/2}e_1+P_{2m-1}\in R_{2m-1}[e_1].
\end{equation} 
For $0<i<m$ define
$$P_{2m-i}:=e_{2m-i}+\sum_{j=1}^i e_jP_{m-i,m-j}^{q^j/2}$$
where $P_{m-i,m-j}\in R_{2m-2j-1}$ as defined in \cite[Proposition 8.3.7]{benson-polyinvafinigrou:93}).
Taking the $m-1+i$ relation in $S_{m-1}^{\syp{2m-2}{q}}$, multiplying by $u_{m-1}$, and interpreting the result as a relation in
$R_{2m-2}$, we see that the coefficient of $z$ in $P_{2m-i}$ is zero. By comparing degrees, we see that $P_{2m-i}\in R_{2m-1}$.
This gives 
\begin{equation} \label{Oodd-exrelgen}
e_{2m-i}=\sum_{j=1}^i e_jP_{m-i,m-j}^{q^j/2}+P_{2m-i}\in R_{2m-1}[e_1,e_2,\ldots,e_i].
\end{equation} 

Define a $(2m-2)\times (2m-2)$ matrix with entries in $R_{2m-3}$ by
\begingroup
\setlength\arraycolsep{2pt}
$$\widetilde{M}_m:=\begin{pmatrix}
      0& \xi_{1}^{q/2}& \xi_{2} ^{q/2}     &\xi_3 ^{q/2}  &\xi_4^{q/2}&& \cdots && \xi_{2m-3}^{q/2}\\
\xi_1^{q/2}& 0         & \xi_{1}^{q^2/2}  & \xi_{2}^{q^2/2}   &  \xi_3^{q^2/2}&& \cdots && \xi_{2m-4}^{q^2/2}\\
\xi_2^{q/2}& \xi_1^{q^2/2}& 0            &\xi_1^{q^3/2} &\xi_2^{q^3/2} && \cdots && \xi_{2m-5}^{q^3/2}\\
&&&&\vdots&&&& \\
\xi_{m-2}^{q/2}&\xi_{m-3}^{q^2/2}&\cdots &\xi_1^{q^{m-2}/2}&0&\xi_1^{q^{m-1}/2}&\xi_2^{q^{m-1}/2}&\cdots &\xi_{m-1}^{q^{m-1}/2}\\
1&0&0&0&& \cdots &&0& 0\\
0&1&0&0&&\cdots &0& 0&P_{m-2,m-2}^{q^2/2}\\
0&0&1&0&0&\cdots &0& P_{m-3,m-3}^{q^3/2}&P_{m-3,m-2}^{q^2/2}\\
&&&&\vdots&&&& \\
0&\cdots&0&1&0&P_{1,1}^{q^{m-1}/2}& P_{1,2}^{q^{m-2}/2}&\cdots&P_{1,m-2}^{q^2/2}
\end{pmatrix}.
$$ 
\endgroup
Using Theorem~\ref{Oodd_det_thm}, we see that $\det(\widetilde{M}_m)=u_{m-1}^{q/2}$.
Equations~\ref{Oodd-eqn1}, \ref{Oodd-eqni}, \ref{Oodd-exrel1} and \ref{Oodd-exrelgen}
can be written in matrix form as 
\begin{equation}\label{Oodd-matrix-eqn}
\widetilde{M}_m \begin{pmatrix} e_{2m-1}\\ e_{2m-2} \\ \vdots\\ e_2 \end{pmatrix} 
=\begin{pmatrix} \xi_{2m-2}^{q/2}e_1+\xi_{2m-1}^{q/2}+u_m^{q/2-1}\xi_1u_{m-1}^{q^2/2}\\ 
\xi_{2m-3}^{q^2/2}e_1+\xi_{2m-2}^{q^2/2}+u_m^{q/2-1}\xi_2\overline{u_{m-1}d_{2m-3}}^{q/2} \\ 
\vdots\\
\xi_m^{q^{m-1}/2}e_1+\xi_{m+1}^{q^{m-1}/2}+u_m^{q/2-1}\xi_{m-1}\overline{u_{m-1}d_m}^{q/2} \\ 
u_{m-1}^{q(q-1)/2}e_1+P_{2m-1} \\
P^{q/2}_{m-2,m-1} e_1+P_{2m-2}\\
\vdots\\
P_{1,m-1}^{q/2} e_1+P_{m+1}\end{pmatrix}.
\end{equation}
Note that the entries on the right hand side of this equation lie in $R_{2m-1}[e_1]$.

\section{Unique Factorisation}\label{sec: ufd}

In this section we complete the computation of $S_m[z]^{\orth{2m+1}{q}}$ for $q$ even with $q>2$.
Using Theorem~\ref{ortho_hsop_thm_char2}, 
    $$
        \HH=\{\xi_0,\xi_1,\ldots,\xi_{m},e_1,e_2,\ldots,e_{m}\}
    $$
is a homogeneous system of parameters.

Let $A$ denote the subalgebra of $S_m[z]^{\orth{2m+1}{q}}$ generated by 
    $$
        \HH\cup\{\xi_{m+1},\xi_{m+2},\ldots,\xi_{2m-1},e_{m+1},e_{m+2},\ldots,e_{2m-1}\}\, .
    $$
It follows from Lemma~\ref{ei-lem} and the definition of $e_1$ that $S_m^{\syp{2m}{q}}\subset A$.
Thus $R_{2m}\subset A$ and, using Theorem~\ref{thm_FF-char2}, $\FF(A)=\FF(S_m[z]^{\orth{2m+1}{q}})$.
Therefore, to show that $A=S_m[z]^{\orth{2m+1}{q}}$, it is sufficient to show that $A$ is integrally closed in its field of fractions.
Since $\det(\widetilde{M}_m)=u_{m-1}^{q/2}$, we can use Equation~\ref{Oodd-matrix-eqn} to solve for $e_i$, with $2\leq i\leq 2m-1$,
as an element of $R_{2m-1}[e_1,u_{m-1}^{-1}]$.
Thus $A[u_{m-1}^{-1}]=R_{2m-1}[\xi_0,e_1,u_{m-1}^{-1}]$.
Since $R_{2m-1}[\xi_0,e_1]$ is a polynomial algebra it is a UFD. Therefore $A$ is a UFD if $u_{m-1}$ is prime in $A$, see 
\cite[Theorem 20.2]{matsumura-commringtheo:86}.

\begin{lemma} \label{Oodd-alg-ind} The set $\{e_1,e_2,\xi_0,\xi_1,\ldots,\xi_{2m-1}\}\setminus\{\xi_{2m-3}\}$ is algebraically independent.
\end{lemma}
\begin{proof} Let $L$ denote the field generated by $\{e_1,e_2,\xi_0,\xi_1,\ldots,\xi_{2m-1}\}\setminus\{\xi_{2m-3}\}$ and
let $K$ denote the field generated by $\{d_{1,m},d_{2,m},\xi_1,\ldots,\xi_{2m-1}\}\setminus\{\xi_{2m-3}\}$. 
We will show that the transcendence  degree of $K$ is $2m$.
Since $K(\xi_0)$ has transcendence  degree $1$ over $K$ and $L$ is a finite extension of $K(\xi_0)$,
this shows that $L$ has transcendence degree $2m+1$, proving the result.

We will use the expressions for $u_m d_{1,m}$ and $u_md_{2,m}$ as elements in $R_{2m}$ given by Lemma~\ref{nat_mon_lem}.
Since the expression for $u_m d_{1,m}$ has degree $1$ as a polynomial in $\xi_{2m}$, 
$K(\xi_{2m-3})=\field_q(\xi_1,\ldots,\xi_{2m})=\field_q(\xi_1,\ldots,\xi_{2m-1},d_{1,m})$
has transcendence degree $2m$. We will show that $K\subset K(\xi_{2m-3})$ is a finite extension.

Cross multiplying to eliminate $\xi_{2m}$, define 
    $$
        F:=\overline{u_{m-1} d_1}^q u_m d_{1,m} -u_{m-1}^q u_m d_{2,m}\in R_{2m-1}\, .
    $$
Dividing by $u_m$ gives $\overline{u_{m-1} d_1}^q d_{1,m} -u_{m-1}^q d_{2,m}=F/u_m\in R_{2m-1}$ (see \cite[Lemma 8.3.5]{benson-polyinvafinigrou:93}).

Define $\overline{F}(t)\in(R_{2m-1}/\langle\xi_{2m-3}\rangle)[t]$ so that $\overline{F}(\xi_{2m-3})=F/u_m$.
Using Lemma~\ref{nat_mon_lem}, $u_{m-1}=\xi_{2m-3}u_{m-2}^q+\varepsilon_1$ and
$$\overline{u_{m-1}d_1}=\xi_{2m-2}u_{m-2}^q+\gamma\xi_{2m-3}^q+\varepsilon_2$$
with $\gamma,\varepsilon_1,\varepsilon_2\in R_{2m-4}\setminus \{0\}$.
Define
$$H(t):=(\xi_{2m-2}u_{m-2}^q+\gamma t^q+\varepsilon_2)^qd_{1,m}+(t u_{m-2}^q+\varepsilon_1)^qd_{2,m}+\overline{F}(t)\in K[t]$$
and observe that $H(\xi_{2m-3})=0$. Therefore, as long as $H$ is not identically zero, $\xi_{2m-3}$ is a root of a polynomial in $K[t]$ and
the field extension $K\subset K(\xi_{2m-3})$ is finite. To see that $H$ is not identically zero, note that the coefficient of $t^{q^2}$
is $\gamma^qd_{1,m}+\varepsilon$ for some 
$\varepsilon\in R_{2m-1}/\langle\xi_{2m-3}\rangle$ and
$\{d_{1,m},\xi_1,\ldots,\xi_{2m-1}\}\setminus\{\xi_{2m-3}\}$ is algebraically independent.
\end{proof}

\begin{lemma}\label{um-lt-lem} $u_m\equiv \xi_m^{1+q+\cdots+q^{m-1}} \mod{\langle \xi_1,\xi_2,\ldots,\xi_{m-1}\rangle}$.
\end{lemma}
\begin{proof} Using Lemma~\ref{nat_mon_lem}, $\xi_m^{1+q+\cdots+q^{m-1}} $ appears in $u_m$.
Furthermore, all of the other terms in $u_m$ include a factor of $\xi_i$ for some $i$ less than $m$.
\end{proof}

\begin{theorem}\label{main_theorem} (a) $A=S_m[z]^{\orth{2m+1}{q}}$.\\
(b) $S_m[z]^{\orth{2m+1}{q}}$ is a complete intersection with relations given by 
Equations~\ref{Oodd-eqn1}, \ref{Oodd-eqni}, \ref{Oodd-exrel1} and \ref{Oodd-exrelgen}.\\
(c) $S_m[z]^{\orth{2m+1}{q}}$ is a free module over $\field_q[\H]$ with module generators given by the monomial factors of 
$\prod_{i=1}^{m-1} \xi_{2m-i}^{(q^i/2)-1}$.
\end{theorem}

\begin{proof} Let $T$ denote the polynomial algebra $R_{2m-1}[\xi_0,E_1,E_2,\ldots,E_{2m-1}]$ and let $\rho:T\to A$ denote the algebra epimorphism taking $E_i$ to $e_i$.  Define 
    $$
        r_1^{(1)}:=\xi_{2m-1}^{q/2}+\xi_{2m-2}^{q/2}E_1+\xi_{2m-3}^{q/2}E_2+\cdots \xi_{1}^{q/2}E_{2m-2}
        +\xi_1u_{m-1}^{q^2/2}u_m^{q/2-1}
    $$
and, for $1<i<m-1$,
    $$
        r_i^{(1)}:=\xi_{2m-i}^{q^i/2}+\sum_{j=1}^{2m-1-i} \xi_{2m-i-j}^{q^i/2} E_j +\sum_{k=1}^{i-1} \xi_k^{q^{i-k}/2} E_{2m-i+k}
        +\xi_i \overline{u_{m-1}d_{2m-1-i}}^{q/2} u_m^{q/2-1}.
    $$
Further define
    $$
        r_1^{(2)}:=E_{2m-1}+u_{m-1}^{(q^2-q)/2}E_1+P_{2m-1} 
    $$
and, for $1<i<m-1$,
    $$
        r_i^{(2)}:=E_{2m-i}+\sum_{j=1}^i E_jP_{m-i,m-j}^{q^j/2}+P_{2m-i}.
    $$
Let $\overline{T}$ denote the quotient $T/\langle r_i^{(1)},r_i^{(2)}\mid i=1,\dots,m-1\rangle$
and observe that $\rho$ induces an epimorphism  from $\overline{T}$ to $A$, say $\overline{\rho}$.
In the following we identify $R_{2m-1}$ with its image in $\overline{T}$ and use $\overline{E}_i$
to denote the equivalence class of $E_i$ in $\overline{T}$.
Note that $$r_i^{(1)}\equiv \xi_{2m-i}^{q^i/2}\mod \langle\xi_1,\ldots,\xi_{2m-i-1}\rangle$$
and $$r_i^{(2)}\equiv E_{2m-i}\mod \langle\xi_1,\ldots,\xi_{2m-1}\rangle.$$
Using Lemma~\ref{um-lt-lem}, 
$$u_{m-1}\equiv \xi_{m-1}^{(q^{m-1}-1)/(q-1)} \mod  \langle\xi_1,\ldots,\xi_{m-2}\rangle$$
and
$$u_{m-2}\equiv \xi_{m-2}^{(q^{m-2}-1)/(q-1)} \mod  \langle\xi_1,\ldots,\xi_{m-3}\rangle.$$
In the polynomial ring $T$ a partial homogeneous system of parameters is a regular sequence.
If we take the regular sequence given by the variables and replace $\xi_{m-2}$ with $u_{m-2}$,
$\xi_{m-1}$ with $u_{m-1}$, $\xi_{2m-i}$ with $r_i^{(1)}$ and $E_{2m-i}$ with $r_i^{(2)}$, we get a new regular sequence.
From this we see that $\overline{T}$ as a complete intersection and $u_{m-2}$, $u_{m-1}$ is a regular sequence in $\overline{T}$.

Using Equation~\ref{Oodd-matrix-eqn}, in $\overline{T}$, we have
$$\widetilde{M}_m \begin{pmatrix} \overline{E}_{2m-1}\\ \overline{E}_{2m-2} \\ \vdots\\ \overline{E}_2 \end{pmatrix} \in \left(R_{2m-1}[\overline{E}_1]\right)^{2m-2}.$$
Since $\det(\widetilde{M}_m)=u_{m-1}^{q/2}$, we can use this to solve for $\overline{E}_i$ in $R_{2m-1}[\overline{E}_1,u_{m-1}^{-1}]$ when $i>1$.
Therefore $\overline{T}[u_{m-1}^{-1}]=R_{2m-1}[\xi_0,\overline{E}_1, u_{m-1}^{-1}]$.
Since $R_{2m-1}[\xi_0,e_1]$ is a polynomial ring, $R_{2m-1}[\xi_0,\overline{E}_1]$ is a polynomial ring, and $\overline{T}$ is a UFD if $u_{m-1}$ is prime in $\overline{T}$.

From Lemma~\ref{nat_mon_lem}, we see that $u_{m-1}-\xi_{2m-3}u_{m-2}^q\in R_{2m-4}$.
Using this, we can solve for $\xi_{2m-3}$ in $(\overline{T}/\langle u_{m-1}\rangle)[u_{m-2}^{-1}]$.
(Note that, since $u_{m-1}$, $u_{m-2}$ is a regular sequence in $\overline{T}$, $u_{m-2}$ is not a zero-divisor in $\overline{T}/\langle u_{m-1}\rangle$.)
It follows from Lemma~\ref{Oodd-alg-ind} that $(R_{2m-1}/\langle \xi_{2m-3}\rangle)[e_1,e_2]$ is a polynomial ring.
We will show that $(\overline{T}/\langle u_{m-1}\rangle)[u_{m-2}^{-1}]$ is isomorphic to $(R_{2m-1}/\langle \xi_{2m-3}\rangle)[e_1,e_2,u_{m-2}^{-1}]$.
It follows from this that $u_{m-1}$ is prime in $\overline{T}$. 

Using $r_2^{(1)}$, we have $\overline{E}_{2m-1}=u_{m-1}^{(q^2-q)/2}\overline{E}_1+P_{2m-1}$, which we use to eliminate $\overline{E}_{2m-1}$ in $\overline{T}$.
Let $\overline{M}$ denote the $(2m-4)\times(2m-4)$ matrix formed from $\widetilde{M}_m$ by removing row $1$, row $m$, column $1$ and column $2m-2$.
Using Theorem~\ref{Oodd_det_thm}, $\det(\overline{M})=u_{m-2}^{q^2/2}$.
We can write Equations~\ref{Oodd-eqni} and \ref{Oodd-exrelgen} in matrix form 
$$\overline{M}\begin{pmatrix} e_{2m-2}\\e_{2m-3} \\ \vdots\\ e_3 \end{pmatrix} =\overline{v}$$
with
    $$
        \overline{v}:=\begin{pmatrix} \xi_1^{q/2}e_{2m-1}+\xi_{2m-4}^{q^2/2}e_2
        +\xi_{2m-3}^{q^2/2}e_1+\xi_{2m-2}^{q^2/2}+u_m^{q/2-1}\xi_2\overline{u_{m-1}d_{2m-3}}^{q/2} \\ 
        \vdots\\
        \xi_{m-2}^{q/2}e_{2m-1}+\xi_{m-1}^{q^{m-1}/2}e_2
        +\xi_m^{q^{m-1}/2}e_1+\xi_{m+1}^{q^{m-1}/2}+u_m^{q/2-1}\xi_{m-1}\overline{u_{m-1}d_m}^{q/2} \\ 
        P^{q^2/2}_{m-2,m-2}e_2+P^{q/2}_{m-2,m-1} e_1+P_{2m-2}\\
        P^{q^2/2}_{m-3,m-2}e_2+P^{q/2}_{m-3,m-1} e_1+P_{2m-3}\\
        \vdots\\
        P_{1,m-2}^{q/2} e_2+P_{1,m-1}^{q/2} e_1+P_{m+1}\end{pmatrix}\, .
    $$
We can use these equations to eliminate $\overline{E}_i$ in $(\overline{T}/\langle u_{m-1}\rangle)[u_{m-2}^{-1}]$ when $2<i<2m-1$.
From this we conclude that $(\overline{T}/\langle u_{m-1}\rangle)[u_{m-2}^{-1}]$ is isomorphic to $$(R_{2m-1}/\langle \xi_{2m-3}\rangle)[e_1,e_2,u_{m-2}^{-1}].$$

We have shown that $\overline{T}$ is integrally closed in its field of fractions.
Since $\overline{\rho}$ induces an isomorphism on fraction fields, it is injective.
Thus $\overline{T}$ is isomorphic to $A$, proving that $A$ is integrally closed in its field of fractions and, therefore,
 $A=S_m[z]^{\orth{2m+1}{q}}$. Furthermore, since $\overline{T}$ is a complete intersection, $S_m[z]^{\orth{2m+1}{q}}$ is a complete intersection,
 completing the proof of parts (a) and (b). Since $S_m[z]^{\orth{2m+1}{q}}$ is a complete intersection, it is Cohen-Macaulay. Thus 
 $S_m[z]^{\orth{2m+1}{q}}$  is a free module over $\field_q[\H]$ of rank
    $$
        \frac{2\prod_{i=1}^m(q^i+1)\prod_{j=1}^m\frac{1}{2}(q^j-1)q^{2m-j}}{q^{m^2}\prod_{k=1}^m(q^{2k}-1)}=\prod_{i=1}^{m-1}\frac{q^i}{2}\, .
    $$
Since the monomial factors of $\prod_{i=1}^{m-1} \xi_{2m-i}^{q^i/2-1}$ are a spanning set of size $\prod_{i=1}^{m-1}\frac{q^i}{2}$, they form a basis, proving part (c).
\end{proof} 

\begin{remark} It follows from part (c) of Theorem~\ref{main_theorem} that
$\H\cup\{\xi_{m+1},\ldots,\xi_{2m-1}\}$ is a minimal generating set for
$S_m[z]^{\orth{2m+1}{q}}$. The relations among these minimal generators can be constructed from Equation~\ref{Oodd-matrix-eqn} by substitution.
\end{remark}

The Steenrod algebra is the $\field_q$-algebra generated by the Steenrod operations subject to the Adem relations, see 
\cite[\S 8.2]{neusel+smith-invatheofinigrou:02}.

\begin{cor}
    $S[z]^{\orth{2m+1}{q}}$ is generated by $\{\xi_0,e_1\}$ as an algebra over the Steenrod algebra.
\end{cor}

\begin{remark} Lemma~\ref{del_rel}, which leads to Equations \ref{Oodd-eqn1}
and \ref{Oodd-eqni}, requires $q>2$. We believe that analogous equations hold for $q=2$, see Remark~\ref{O7-rem} for the $m=3$ case. Deriving these equations would give an alternative computation of $S[z]^{\orth{2m+1}{2}}$ to the one given in 
\cite[Theorem~ 6.1]{Kropholler+Rajaei+Segal:05}.
\end{remark}

\section*{Acknowledgement}
  The computer algebra package Magma \cite{magma:97} was very useful in this work.  It was
used to test and confirm many hypotheses in low dimensions for small $q$.
We are grateful to Karl Dilcher and Keith Taylor (Canadian Mathematical Society book editors) and Donna Chernyk (Springer) for encouraging us to work on this problem.

\providecommand{\MR}{\relax\ifhmode\unskip\space\fi MR }


\begin{thebibliography}{10}

\bibitem{benson-polyinvafinigrou:93}
D.~J. Benson, \emph{Polynomial invariants of finite groups}, London
  Mathematical Society Lecture Note Series, vol. 190, Cambridge University
  Press, Cambridge, 1993. \MR{94j:13003}

\bibitem{magma:97}
W.~Bosma, J.~Cannon, and C.~Playoust,\emph{The Magma algebra system. I. The user language},
Computational algebra and number theory, J. Symbolic Comput., 24, No. 3-4, 1997, 235--265,
\MR{1484478} 

\bibitem{campbell+chuai-local:07}
H~E~A Campbell and Jianjun Chuai, \emph{On the invariant fields and localized
  invariant rings of $p$-groups}, Quarterly Journal of Mathematics \textbf{10}
  (2007), no.~2, 1--7. \MR{2334859}


\bibitem{campbell+shank+wehlau:24}
H. E. A. Campbell, R. J. Shank, and D. L. Wehlau, \emph{Invariants of finite orthogonal groups of
plus type in odd characteristic}, arXiv:2407:01152 (2024), 1–36.


\bibitem{campbell+wehlau:mit11}
H.~E. A. Campbell and David~L. Wehlau, \emph{Modular invariant theory},
  Encyclopaedia of Mathematical Sciences, vol. 139, Springer-Verlag, Berlin,
  2011, Invariant Theory and Algebraic Transformation Groups, 8. \MR{2759466}


\bibitem{chiang+hung-invorth:93}
Li~Chiang and Yu~Ch'ing Hung, \emph{The invariants of orthogonal group
  actions}, Bull. Austral. Math. Soc. \textbf{48} (1993), no.~2, 313--319.
  \MR{1238804}

\bibitem{chu-polyinvfourdimorth:01}
Huah Chu, \emph{Polynomial invariants of four-dimensional orthogonal groups},
  Comm. Algebra \textbf{29} (2001), no.~3, 1153--1164. \MR{1842403}

\bibitem{chu+jow-unitary:06}
Huah Chu and Shin-Yao Jow, \emph{Polynomial invariants of finite unitary
  groups}, J. Algebra \textbf{302} (2006), no.~2, 686--719. \MR{2293777}

\bibitem{cohen-ratfuncortho:90}
S.~D. Cohen, \emph{Rational functions invariant under an orthogonal group},
  Bull. London Math. Soc. \textbf{22} (1990), no.~3, 217--221. \MR{1041133}

  
\bibitem{derksen+kemper-invalggp:08}
  Harm Derksen and Gregor Kemper, \emph{Computing invariants of algebraic groups in arbitrary
characteristic}, Adv. Math. 217 (2008), no. 5, 2089–2129. MR 2388087

\bibitem{dickson-fundsystinvagene:11}
Leonard~Eugene Dickson, \emph{A fundamental system of invariants of the general
  modular linear group with a solution of the form problem}, Trans. Amer. Math.
  Soc. \textbf{12} (1911), no.~1, 75--98. 
  \MR{1500882}



\bibitem{Kropholler+Rajaei+Segal:05}
P.~H. Kropholler, S.~Mohseni~Rajaei, and J.~Segal, \emph{Invariant rings of
  orthogonal groups over {$\mathbb F_2$}}, Glasg. Math. J. \textbf{47} (2005),
  no.~1, 7--54. \MR{2200953}


\bibitem{matsumura-commringtheo:86}
Hideyuki Matsumura, \emph{Commutative ring theory}, Cambridge Studies in
  Advanced Mathematics, vol.~8, Cambridge University Press, Cambridge, 1986.
  \MR{88h:13001}


\bibitem{neusel+smith-invatheofinigrou:02}
Mara~D. Neusel and Larry Smith, \emph{Invariant theory of finite groups},
  Mathematical Surveys and Monographs, vol.~94, American Mathematical Society,
  Providence, RI, 2002. \MR{1869812}

\bibitem{smith-ringinvorth:99}
Larry Smith, \emph{The ring of invariants of {${\rm O}(3,{\bf F}_q)$}}, Finite
  Fields Appl. \textbf{5} (1999), no.~1, 96--101. \MR{1667106}

\bibitem{stanley-invafinigrouthei:79}
    Stanley, Richard P., \emph{Invariants of finite groups and their applications to combinatorics},
    Bull. Amer. Math. Soc. (N.S.)
    \textbf{1} (1979), no.~1, 475--511. \MR{526968}
 
\bibitem{Taylor-ClassicalGroups:92}
Donald~E. Taylor, \emph{The geometry of the classical groups}, Sigma Series in
  Pure Mathematics, vol.~9, Heldermann Verlag, Berlin, 1992. \MR{1189139}

\bibitem{wilkerson-primdickinva:83}
Clarence Wilkerson, \emph{A primer on the {D}ickson invariants}, Proceedings of
  the Northwestern Homotopy Theory Conference (Evanston, Ill., 1982)
  (Providence, RI), Contemp. Math., vol.~19, Amer. Math. Soc., 1983,
  pp.~421--434. \MR{85c:55017}
\end{thebibliography}
\end{document}